\documentclass[12pt]{amsart}

\usepackage{amsmath,amsthm,verbatim,amssymb,amsfonts,amscd, graphicx,enumitem}
\usepackage{graphics}
\usepackage{tikz-cd}
\usepackage{pinlabel}
\usepackage{xcolor}
\usepackage{stackrel}
\usepackage{caption}
\usepackage[labelformat=simple]{subcaption}
\usepackage{hyperref}

\usepackage[pagewise]{lineno}

\usepackage[margin=1.5in]{geometry}
\newtheorem{theorem}{Theorem}[section]
\newtheorem{corollary}[theorem]{Corollary}
\newtheorem{lemma}[theorem]{Lemma}
\newtheorem{proposition}[theorem]{Proposition}

\newtheorem*{cor:euler}{Corollary \ref{cor:euler}}
\newtheorem*{cor:orient0}{Corollary \ref{cor:orient0}}

\newtheorem*{cor:differentbridge}{Corollary \ref{cor:differentbridge}}

\makeatletter
\newtheorem*{rep@theorem}{\rep@title}
\newcommand{\newreptheorem}[2]{%
\newenvironment{rep#1}[1]{%
 \def\rep@title{#2 \ref{##1}}%
 \begin{rep@theorem}}%
 {\end{rep@theorem}}}
\makeatother

\newtheorem{question}[theorem]{Question} 
\newtheorem{claim}[theorem]{Claim} 
\newtheorem{questions}[theorem]{Questions}
\theoremstyle{definition}
\newtheorem{definition}[theorem]{Definition}
\newtheorem{remark}[theorem]{Remark}
\newtheorem{remarks}[theorem]{Remarks}

\newtheorem*{observation*}{Observation}

\newreptheorem{theorem}{Theorem}
\newreptheorem{lemma}{Lemma}
\newreptheorem{question}{Question}
\newreptheorem{corollary}{Corollary}
\newreptheorem{proposition}{Proposition}

\theoremstyle{definition}
\newtheorem{example}{Example}[section]

\newcommand{\leftrarrows}{\mathrel{\raise.75ex\hbox{\oalign{%
  $\scriptstyle\leftarrow$\cr
  \vrule width0pt height.5ex$\hfil\scriptstyle\relbar$\cr}}}}
\newcommand{\lrightarrows}{\mathrel{\raise.75ex\hbox{\oalign{%
  $\scriptstyle\relbar$\hfil\cr
  $\scriptstyle\vrule width0pt height.5ex\smash\rightarrow$\cr}}}}
\newcommand{\Rrelbar}{\mathrel{\raise.75ex\hbox{\oalign{%
  $\scriptstyle\relbar$\cr
  \vrule width0pt height.5ex$\scriptstyle\relbar$}}}}

\makeatletter
\def\leftrightarrowsfill@{\arrowfill@\leftrarrows\Rrelbar\lrightarrows}
\newcommand{\xleftrightarrows}[2][]{\ext@arrow 3399\leftrightarrowsfill@{#1}{#2}}
\makeatother

\newcommand{\Z}{\mathbb{Z}}
\newcommand{\R}{\mathbb{R}}

\newcommand{\CP}{\mathbb{CP}}
\newcommand{\DD}{\mathbb{D}}

\newcommand{\TT}{\mathfrak{T}}
\newcommand{\Ss}{\mathcal{S}}
\newcommand{\D}{\mathcal{D}}

\newcommand{\T}{\mathcal{T}}
\newcommand{\sS}{\mathcal{S}}

\newcommand{\Ll}{\mathcal{L}}
\newcommand{\Rr}{\mathcal{R}}

\newcommand{\be}{\begin{enumerate}}
\newcommand{\ee}{\end{enumerate}}

\let\int\relax
\newcommand{\int}{\mathring}

\newcommand{\boundary}{\partial}

\DeclareMathOperator{\tube}{{Tube}}

\usepackage{color}
\definecolor{darkpurple}{rgb}{.5,0,.5}

\begin{document}

\title{Bridge trisections and classical knotted surface theory}

\author{Jason Joseph}
\address{Department of Mathematics\\Rice University\\Houston, TX, USA}
\email{jason.joseph@rice.edu}

\author{Jeffrey Meier}
\address{Department of Mathematics\\Western Washington University\\Bellingham, WA, USA}
\email{jeffrey.meier@wwu.edu}

\author{Maggie Miller}
\address{Department of Mathematics\\Stanford University\\Stanford, CA, USA}
\email{maggie.miller.math@gmail.com}

\author{Alexander Zupan}
\address{Department of Mathematics\\University of Nebraska-Lincoln\\Lincoln, NE, USA}
\email{zupan@unl.edu}


\begin{abstract}
	We seek to connect ideas in the theory of bridge trisections with other well-studied facets of classical knotted surface theory.
	First, we show how the normal Euler number can be computed from a tri-plane diagram, and we use this to give a trisection-theoretic proof of the Whitney--Massey Theorem, which bounds the possible values of this number in terms of the Euler characteristic.
	Second, we describe in detail how to compute the fundamental group and related invariants from a tri-plane diagram, and we use this, together with an analysis of bridge trisections of ribbon surfaces, to produce an infinite family of knotted spheres that admit non-isotopic bridge trisections of minimal complexity.
\end{abstract}

\maketitle

\section{Introduction}\label{sec:intro}

In this paper, we study bridge trisections of surfaces in $S^4$, as originally introduced by the second and fourth authors in \cite{MeiZup_bridge1}. A {\emph{bridge trisection}} of a surface $\Ss$ in $S^4$ is a certain decomposition of $(S^4,\Ss)$ into three trivial disk systems $(B^4_1,\D_1),(B^4_2,\D_2),(B^4_3,\D_3)$, a four-dimensional analog of a bridge splitting, which cuts a link in $S^3$ into two trivial tangles.  The purpose of this paper is to connect the theory of bridge trisections with a number of different ideas and results in classical knotted surface theory.  In particular, we demonstrate how to use a bridge trisection of a surface $\Ss\subset S^4$ to compute various invariants of $\Ss$ and to obtain other topological information.  We give a more precise definition and much relevant background information in Section \ref{sec:prelim}.

In Section \ref{sec:broken}, we describe a method of obtaining a broken surface diagram for $\Ss$ from a tri-plane diagram. Using this method, we can recover the Euler number $e(\Ss)$ of the normal bundle of $\Ss$.

\begin{cor:euler}
Let $\DD=(\DD_1,\DD_2,\DD_3)$ be a tri-plane diagram of a surface $\Ss\subset S^4$. Let $w_i$ be the writhe of the diagram $\DD_i\cup\overline{\DD}_{i+1}$. Then $e(\Ss)=w_1+w_2+w_3$.
\end{cor:euler}

As one application, we obtain the following well-known result.

\begin{cor:orient0}
If $\Ss$ is oriented, then $e(\Ss)=0$.
\end{cor:orient0}

As another application, we deduce a new proof of the Whitney--Massey Theorem~\cite{massey} on the Euler number of a surface in $S^4$.

\begin{reptheorem}{thm:massey}
If $\Ss$ is connected and non-orientable with Euler characteristic $\chi$, then  \[e(\Ss)\in\{2\chi-4,2\chi,2\chi+4,\ldots,-2\chi-4,-2\chi,-2\chi+4\}.\]
\end{reptheorem}

In Section~\ref{sec:pi1}, we describe how to calculate the fundamental group of the complement of~$\Ss$.

\begin{reptheorem}{thm:grp_pres}
	Let $\DD$ be a $(b;\bold c)$--tri-plane diagram for a surface knot $\Ss\subset S^4$.  Then $\pi_1(S^4\setminus\nu(\Ss))$ admits a presentation of each of the following types:
	\begin{enumerate}
		\item\label{firstpres} $2b$ meridional generators and $3b$ Wirtinger relations,
		\item $b$ meridional generators and $2b$ Wirtinger relations, or
		\item\label{thirdpres} $c_i$ meridional generators and $b$ Wirtinger relations (for any $i\in\Z_3$).
	\end{enumerate}
Moreover, these presentations can be obtained explicitly from $\DD$.
\end{reptheorem}

Additionally, we use these presentations to show how one may recover more sophisticated information such as the peripheral subgroup of $\Ss$.

In Subsection \ref{subsec:ribbon}, we show how to construct a bridge trisection from any ribbon presentation of a ribbon surface. These bridge trisections always have triple point number zero.  In Subsection \ref{subsec:Nielsen}, we prove that such a bridge trisection respects the Nielsen class of the original ribbon presentation, yielding the following corollary.

\begin{cor:differentbridge}
There exist infinitely many ribbon 2--knots with pairs of bridge trisections $\TT$ and $\TT'$, both induced by ribbon presentations, which are non-isotopic as bridge trisections.
\end{cor:differentbridge}

We conclude with several questions about ribbon bridge trisections.

\subsection*{Acknowledgements}

This paper began following discussions at the workshop \emph{Unifying 4--Dimensional Knot Theory}, which was hosted by the Banff International Research Station in November 2019, and the authors would like to thank BIRS for providing an ideal space for sparking collaboration.
We are grateful to Masahico Saito for sharing his interest in the Seifert algorithm for bridge trisections and motivating this paper, as well as a subsequent one.
The authors would like to thank Rom\'an Aranda, Scott Carter, and Peter Lambert-Cole for helpful conversations.  JJ was supported by MPIM during part of this project, as well as NSF grants DMS-1664567 and DMS-1745670. JM was supported by NSF grants DMS-1933019 and DMS-2006029. MM was supported by MPIM during part of this project, as well as NSF grants DGE-1656466 (at Princeton) and DMS-2001675 (at MIT) and a research fellowship from the Clay Mathematics Institute (at Stanford).  AZ was supported by MPIM during part of this project, as well as NSF grants DMS-1664578 and DMS-2005518.

\section{Preliminaries}\label{sec:prelim}
We work throughout in the smooth category. We begin by describing the simplest 4--manifold trisection, which is the only one necessary for understanding the bulk of this paper. We refer the reader to \cite{gaykirby} for more information on general trisections.

\subsection{Bridge trisections}

We refer the reader to~\cite{MeiZup_bridge1} for complete details, but give the definition of a bridge trisection here for completeness.

%
Let $S^4 = X_1\cup X_2\cup X_3$ be the standard genus zero trisection. We adopt the orientation conventions that $\Sigma = \partial H_i$ for each $i$ and $\partial X_i = H_i\cup\overline H_{i+1}$.

\begin{definition}[\cite{MeiZup_bridge1}]
Let $\Ss$ be a (smooth) closed surface in $S^4$. We say that $\Ss$ is in {\emph{$(b;\mathbf{c})$--bridge position}}, where $b$ is a positive integer and $\mathbf{c}=(c_1,c_2,c_3)$ is a triple of positive integers, if the following are all true:
\begin{enumerate}
\item For each $i$, $\D_i = X_i\cap \Ss_i$ is a collection of $c_i$ boundary-parallel disks in the 4-ball $X_i$.
\item For each $i$, $\T_i = H_i\cap\Ss$ is a boundary-parallel tangle in the 3-ball $H_i$; and
\item $\Sigma\cap\Ss$ is $2b$ points called {\emph{bridge points}};
\end{enumerate}
We denote $H_i\cap\Ss$ by $\alpha$, $\beta$, and $\gamma$, for $i=1,2,3$, respectively.
\end{definition}

Given $\Ss\subset S^4$ in bridge position, we may refer to the decomposition
$$(S^4,\Ss) = (X_1,\D_1)\cup(X_2,\D_2)\cup(X_3,\D_3)$$
as a {\emph{bridge trisection}} of $\Ss$, which we denote by $\TT$. We say that two bridge trisections $\TT$ and $\TT'$ are \emph{equivalent} if there is a diffeomorphism $\phi\colon S^4\to S^4$ with
$$\phi((X_i,\D_i))=(X_i',\D_i')$$
for all $i\in\Z_3$.
Any (smooth) closed surface in $S^4$ can be isotoped into bridge position, regardless of connectivity, genus, or orientability~\cite{MeiZup_bridge1}.

\subsection{Diagrams for bridge trisections}

Bridge trisections may be viewed as the analog to bridge position of a link in $S^3$. One purpose of this article is to demonstrate that bridge trisections are useful for similar reasons. In particular, the theory produces simple diagrams of surfaces that we will use to perform several different computations.

\begin{definition}
A \emph{tri-plane diagram} is a triple of trivial tangle diagrams $\DD = (\DD_1, \DD_2, \DD_3)$ such that $\DD_i\cup\overline\DD_{i+1}$ is a planar diagram for an unlink.
\end{definition}

If each $\DD_i$ is a $b$-stranded trivial tangle and each $\DD_i\cup\overline\DD_{i+1}$ is a $c_i$--component unlink, then up to isotopy, the tri-plane diagram $\DD = (\DD_1, \DD_2, \DD_3)$ determines a $(b,\bold c)$--bridge trisection of a surface $\Ss$ in the following way:
The tri-plane diagram $\DD$ determines the intersection of $\Ss$ with a regular neighborhood of $H_1 \cup H_2 \cup H_3$.  The remainder of $\Ss$ consists of three systems of boundary parallel disks in the 4--balls $X_i$. But boundary parallel disks in a 4--ball are determined up to isotopy rel-boundary by their boundary (see e.g. \cite{kss,liv}).

\begin{remark}
If two tri-plane diagrams $\DD=(\DD_1,\DD_2,\DD_3)$ and $\DD'=(\DD'_1,\DD'_2,\DD'_3)$ describe isotopic surfaces in $S^4$, then by \cite{MeiZup_bridge1} the diagram $\DD$ can be transformed into $\DD'$ by a sequence of the following moves, illustrated in Figures~4 and~27 of~\cite{MeiZup_bridge1}.
\begin{enumerate}
\item {\emph{Mutual braid transposition}}, in which $\DD_1$, $\DD_2$, and $\DD_3$ are replaced by concatenations $\DD_1\beta$, $\DD_2\beta$, and $\DD_3\beta$ (respectively) for some braid diagram $\beta$, 
\item {\emph{Elementary perturbation and deperturbation}}, as illustrated in generality in Figure~27 of~\cite{MeiZup_bridge1}. This operation increases one component of $\mathbf{c}$; the roles of $X_1,X_2,X_3$ as illustrated may be permuted cyclically,
\item Interior Reidemeister moves.
\end{enumerate}
See Subsection~2.5 and Section~6 of~\cite{MeiZup_bridge1} for complete details regarding these moves.
\end{remark}


%

\begin{remark}\label{rem:rollspun}

In the original \cite{MeiZup_bridge1} construction of bridge trisections, the authors give an algorithm to obtain a tri-plane diagram of a surface $\Ss$ given a {\emph{banded unlink diagram}} $(L,B)$ of $\Ss$. A banded unlink diagram, introduced by \cite{kss}, consists of an unlink $L$ and a set of bands $B$ attached to $L$, with the property that surgering $L$ along $B$ yields another unlink $L_B$. The diagram $(L,B)$ determines the surface $\Ss$ in $S^4=S^3\times I/\sim$ up to smooth ambient isotopy. Here, $\Ss$ consists of the following pieces.
\begin{itemize}
\item A collection of disks bounded by $L$ in $S^3\times 1/4$, 
\item $L\times[1/4,1/2]$,
\item $(L\cup B)\times\{1/2\}$,
\item $L_B\times[1/2,3/4]$,
\item a collection of disks bounded by $L_B$ in $S^3\times 3/4$.
\end{itemize}

If $(L,B)$ is in bridge position with respect to a sphere $F$ splitting $S^3$ into 3-balls $B_1$ and $B_2$ (with bands in $B_2$), then $(\D_1,\D_2,\D_3)$ is a bridge trisection of $\Ss$, with $\D_1,\D_2,\D_3$ diagrams of the tangles $(\overline{B_1},\overline{B_1\cap L})$, $(B_2,B_2\cap L)$, $(B_2,B_2\cap L_B)$, respectively. We refer the reader to \cite{MeiZup_bridge1} for more details, but include Figure \ref{fig:rollspun} to illustrate that one generally expects to find tri-plane diagrams of high complexity. We give a more detailed caption of Figure \ref{fig:rollspun} now.

\begin{enumerate}[label=(\alph*)]
\item
In (a), we draw a band diagram of the standard ribbon disk for $\overline{K}\#K$. This consists of the link $\overline{K}\#K$ (the boundary of the disk) and horizontal bands with the property that $\overline{K}\#K$ surgered along these bands is an unlink. We also draw a torus about the $K$ summand; Litherland describes a homeomorphism $\rho$ of $S^3$ supported near this torus consisting of a Dehn twist about a 0-framed longitude of $K$. The roll-spun knot of $K$ is a knotted sphere obtained by gluing two copies of this disk via the boundary homeomorphism $\rho$.  (See \cite{litherland} for the original construction.)
\item In (b), we apply $\rho$ to the diagram of $(a)$.
\item Combining these diagrams in $(c)$ (where we have dualized the bands in $(a)$ and isotoped $(b)$ to simplify the diagram) yields a banded unlink diagram for the roll-spun knot of $K$.
\item In (d) we further isotope this diagram to make the bands appear small. Since we started with $\overline{K}\# K$ in bridge position in (a), this banded unlink is very nearly in bridge position.
\item In (e), we perturb the diagram slightly near the bands to obtain a banded unlink diagram in bridge position; we include a horizontal line indicating the bridge sphere. The lower 3-ball is $B_1$ and the upper 3-ball is $B_2$.
\item Finally in (f), we obtain a tri-plane diagram for the roll-spun knot of $K$.
\end{enumerate}

\begin{figure}
{\centering
\labellist
\pinlabel {(a)} at 130 565
\pinlabel {(b)} at 398 565
\pinlabel {(c)} at 660 565
\pinlabel {(d)} at 252 363
\pinlabel {(e)} at 550 363
\pinlabel {(f)} at 398 160
\pinlabel {3} at 508 489
\pinlabel {3} at 774 489
\pinlabel {3} at 376 274
\pinlabel {3} at 642 274
\pinlabel {-3} at 235 51
\endlabellist
\includegraphics[width=5.5in]{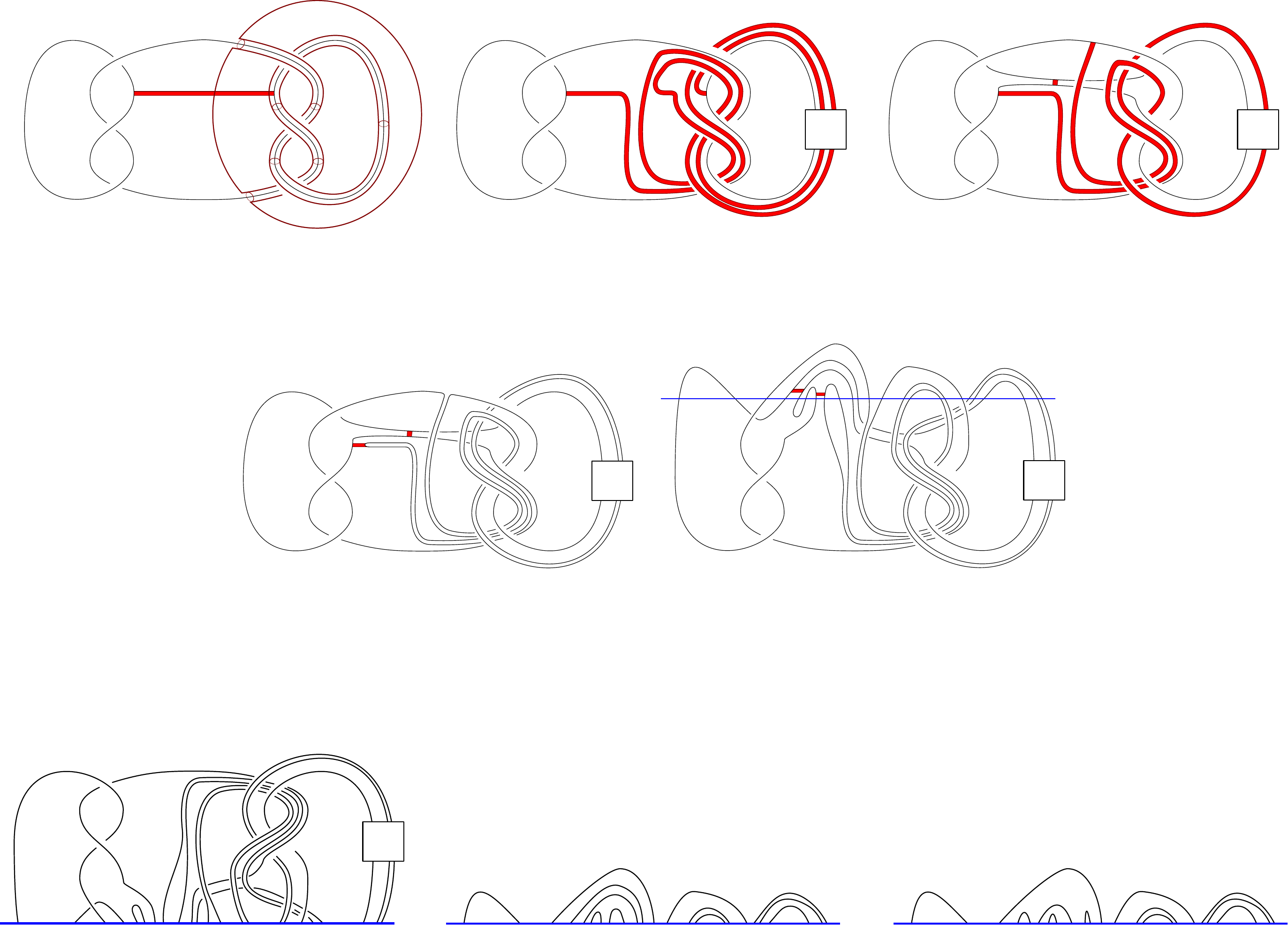}
\vspace{.2in}
\caption{The process of obtaining a tri-plane diagram of a roll-spun knot. We provide a more detailed caption in Remark \ref{rem:rollspun}.}
\label{fig:rollspun}
}
\end{figure}
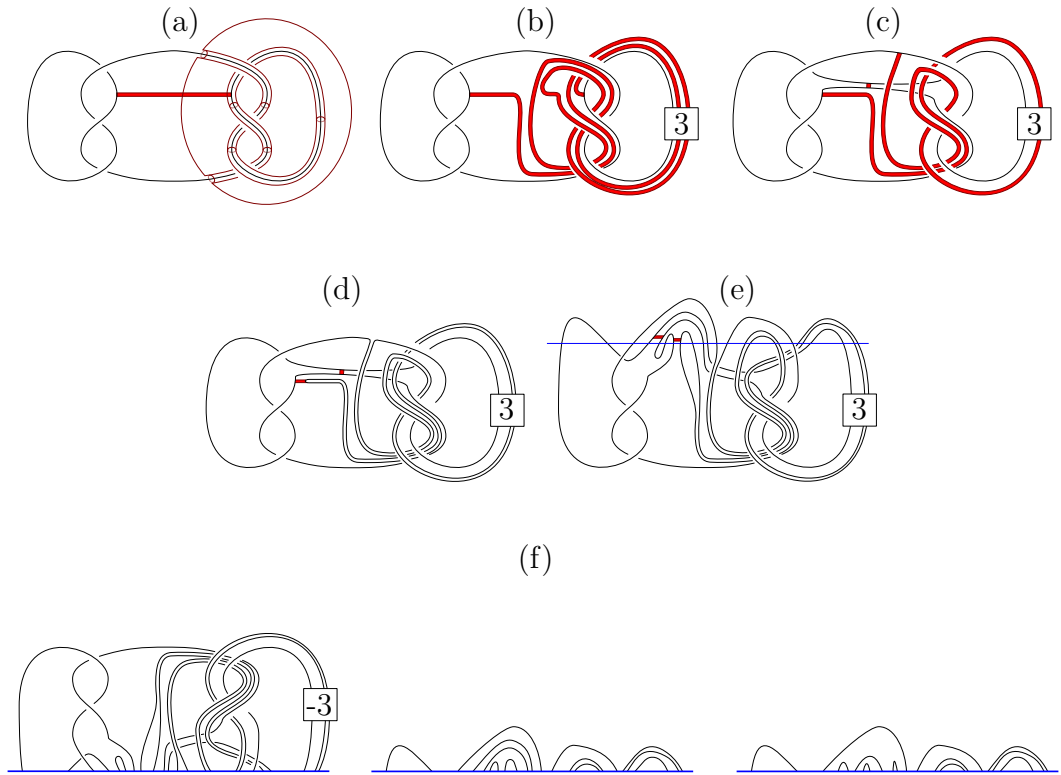
\end{remark}

\subsection{Unknotted surfaces}
\label{subsec:unknotted}

We now recall the notion of unknottedness for surfaces in $S^4$; see~\cite{MTZ_graph} for a related discussion of bridge trisections of unknotted surfaces.

\begin{definition}
Let $\Ss$ be a closed, connected, orientable surface in $S^4$. We say that $\Ss$ is \emph{unknotted} if $\Ss$ bounds an embedded, 3-dimensional handlebody in~$S^4$.

If $\Ss\cong\mathbb{RP}^2$, then we say that $\Ss$ is {\emph{unknotted}} if it is isotopic to one of the two $\mathbb{RP}^2$s in Figure~\ref{fig:unknottedrp2}; we denote these surfaces by $P_\pm$, noting that $e(P_+) = +2$; see Remark~\ref{rmk:sign}. Otherwise, if is $\Ss$ closed, connected, and nonorientable, we say that $\Ss$ is {\emph{unknotted}} if $\Ss$ is isotopic to a connected sum of unknotted $\mathbb{RP}^2$s.

A disconnected surface $\Ss = \Ss_1\cup\cdots\cup \Ss_k$ in  $S^4$ is said to be \emph{unknotted} if there exist disjoint 4--balls $B_1,\ldots, B_k\subset S^4$ with $\Ss_i\subset B_i$, and each $S_i$ is unknotted.
\end{definition}

\begin{figure}[htp]
	\centering
	\includegraphics[width=.8\linewidth]{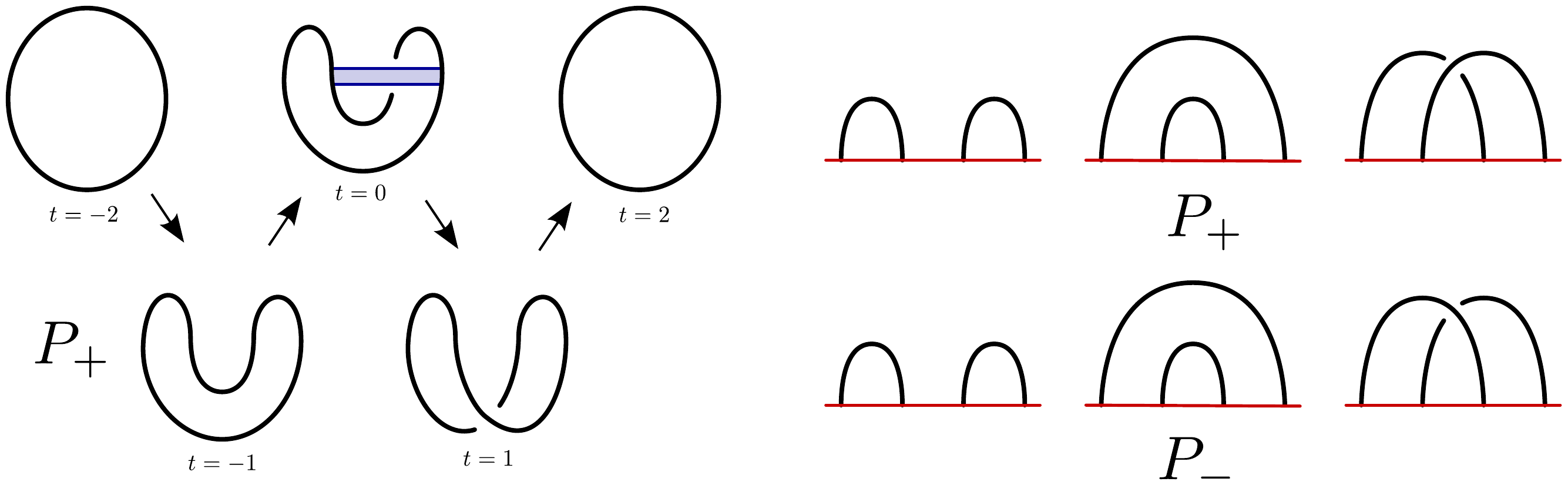}
	\caption{\textbf{Left:} A motion picture describing the unknotted projective plane $P_+$, the unknotted $\mathbb{RP}^2$ with Euler number $+2$.
	\textbf{Top-right:} A tri-plane diagram for $P_+$.
	\textbf{Bottom-right:} A tri-plane diagram for $P_-$.}
	\label{fig:unknottedrp2}
\end{figure}

\begin{remark}
\label{rmk:sign}
	There is a subtlety regarding sign conventions for the unknotted projective planes that is worth noting. Our convention is to denote by $P_+$ the unknotted projective plane with normal Euler number $+2$. As a consequence, we have that the 2--fold cover of $S^4$, branched along $P_+$ is $\overline\CP^2$.
	See Figure~\ref{fig:P-DBC}.
	This is often confused in the literature; for example, it seems that the  orientation of $S^4$ that is adopted in~\cite{Kamada-non} is the opposite of the usual one, leading to the conclusion that $\CP^2$ is the 2--fold cover of $\overline S^4$ branched along $\overline P_+$, which is there denoted ``$P_+$".
	While this seems to be technically correct, it does confuse the issue, slightly.
	In the same vein, Figure~\ref{fig:unknottedrp2} corrects Figure~15 of~\cite{MeiZup_bridge1}, where the motion-picture is mislabeled, though the tri-plane diagrams are correctly labeled.
	Figure~\ref{fig:P-DBC} corrects Figure~2 of~\cite{MeiZup_bridge2}.
	The careful reader will note that when taking branched covers of tri-plane diagrams, it is more natural to revolve a tri-plane diagram 180 degrees, so that the tangles descend from the bridge surface instead of ascending, as shown. (Alternatively, one might view the tangles from below their boundary.)
\end{remark}

\begin{figure}[htp]
	\centering
	\includegraphics[width=.7\linewidth]{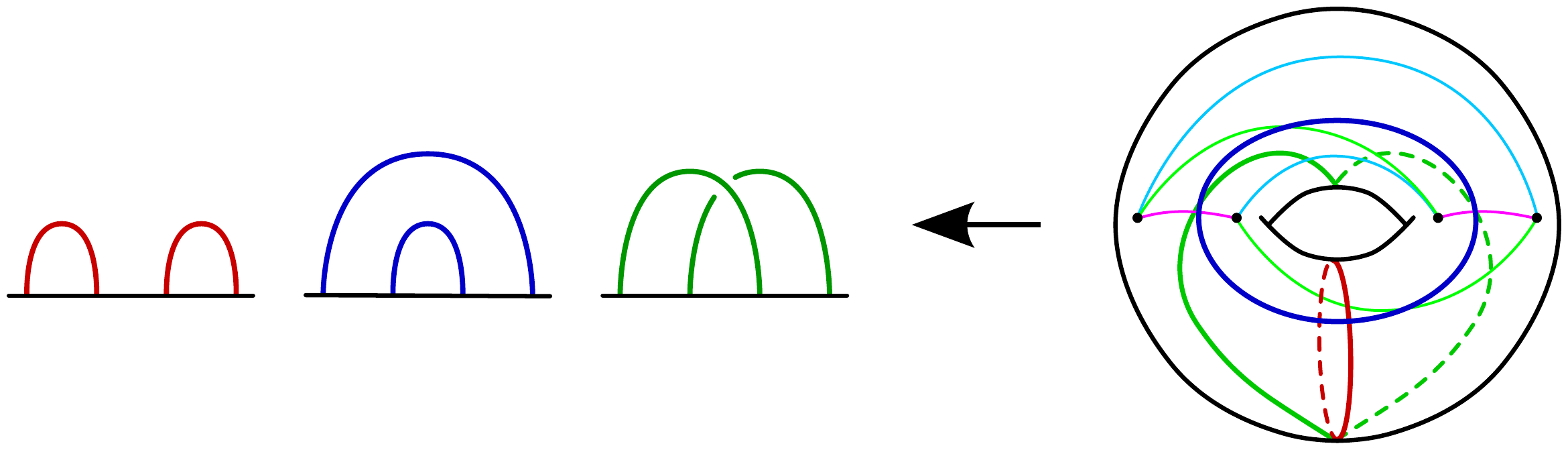}
	\caption{The genus one trisection of $\CP^2$, realized as the 2--fold branched cover of the 2--bridge trisection of $P_-$.}
	\label{fig:P-DBC}
\end{figure}

\section{Broken surface diagrams vs. bridge trisections}\label{sec:broken}

In this section, we discuss the relationship between tri-plane diagrams and broken surface diagrams of  a surface $\Ss$ smoothly embedded in $S^4$.  As result, we obtain a new formula for the normal Euler number $e(\Ss)$ that depends on the writhes of the pairings of tangle diagrams in a tri-plane diagram $\DD$ for $\Ss$, and we give a new proof of the Whitney-Massey Theorem.  We also explore the relationship between ribbon surfaces and bridge trisections.

\subsection{Broken surface diagrams}\label{subsec:broken}
We start by reviewing the notion of a broken surface diagram; see e.g. \cite{CarterSaito} for more exposition.

\begin{definition}

Let $\Ss$ be a surface smoothly embedded in $S^4$. Let $P_S$ be the projection of $\Ss$ to the equatorial $S^3$ of $S^4$. Assume that $P_S$ is generic; i.e. $P_S$ is a smoothly embedded surface away from self-intersections that come in three possible types: arcs of double points, branch points (which necessarily end arcs of double points), and triple points (which are intersections of three arcs of double points). The branch points and triple points are all isolated.

Near each self-intersection of $P_S$, remove a small neighborhood of the intersection from the sheet(s) that is lower in the fourth coordinate of $S^4 = \R^4\cup\infty$.
We call the resulting {\emph{broken surface}} $\mathbb{S}$, as it is embedded in $S^3$, a {\emph{broken surface diagram}} for $\Ss$. This is completely analogous to how one defines a classical knot diagram by indicating over and under information at each crossing.

We will generally refer to $P_S$ as the {\emph{underlying surface}} of $\mathbb{S}$, just as an immersed curve underlies a knot diagram.

\end{definition}





\begin{theorem}\label{thm:getbroken}
Let $\DD=(\DD_1,\DD_2,\DD_3)$ be a tri-plane diagram of a surface $\Ss\subset S^4$. From $\DD$, there is a procedure to produce a broken surface diagram $\mathbb{S}$ of $\Ss$.
\end{theorem}

To prove Theorem \ref{thm:getbroken}, it will be useful to develop some notation to describe simple broken surfaces in $S^3$.

\begin{definition}\label{productsurface}
Let $\DD_L$ be a link diagram in $S^2$. We obtain a broken surface diagram~$\mathbb{L}$ in $S^2\times I$ whose underlying surface is $P_L=D_L\times I$, where $D_L$ is the immersed multicurve underlying $\DD_L$. At self-intersections of $P_L$, the sheet of $\mathbb{L}$ containing the corresponding undercrossing of $\DD\times 0$ is broken. We call $\mathbb{L}$ a {\emph{product}} broken surface, and may write $\mathbb{L}=\DD_L\times I$ as shorthand. We illustrate some product broken surface diagrams (and some non-product diagrams) in Figure \ref{fig:Rmoves}.
\end{definition}

In Definition \ref{productsurface}, we slightly abuse the notation, since $P_L$ is a surface with boundary properly immersed in $S^2\times I$ rather than a closed surface in $S^3$, but this distinction is not important in the setting of this paper.

\begin{remark}
Note that a product broken surface diagram contains only double arcs of intersections. That is, a product broken surface diagram does not include any triple points or branch points.
\end{remark}

\begin{definition}
Let $\DD_L\subset S^2$ be a link diagram. Let $\DD_L'$ be obtained from $\DD_L$ by a single Reidemeister move $R$. We obtain a broken surface diagram $\mathbb{T}_R$ in $ S^2\times I$ whose boundary is $\overline{(\DD_L'\times 0)}\sqcup (\DD_L\times 1)$ which agrees with the product $\DD_L\times I$ away from the support of $R$, and, near $R$, agrees with the corresponding diagram in Figure \ref{fig:Rmoves}. We call $\mathbb{T}_R$ the {\emph{trace}} of $R$.

If $\DD_J$ is obtained from $\DD_L$ by a sequence $\Rr = (R_1,\ldots, R_n)$ of Reidemeister moves, then we write $\mathbb{T}_\Rr$ to denote the concatenation of $\mathbb{T}_{R_1},\ldots,\mathbb{T}_{R_n}$. We call $\mathbb{T}_\Rr$  the \emph{trace} of $\Rr$.
\end{definition}

\begin{figure}
\includegraphics[width=85mm]{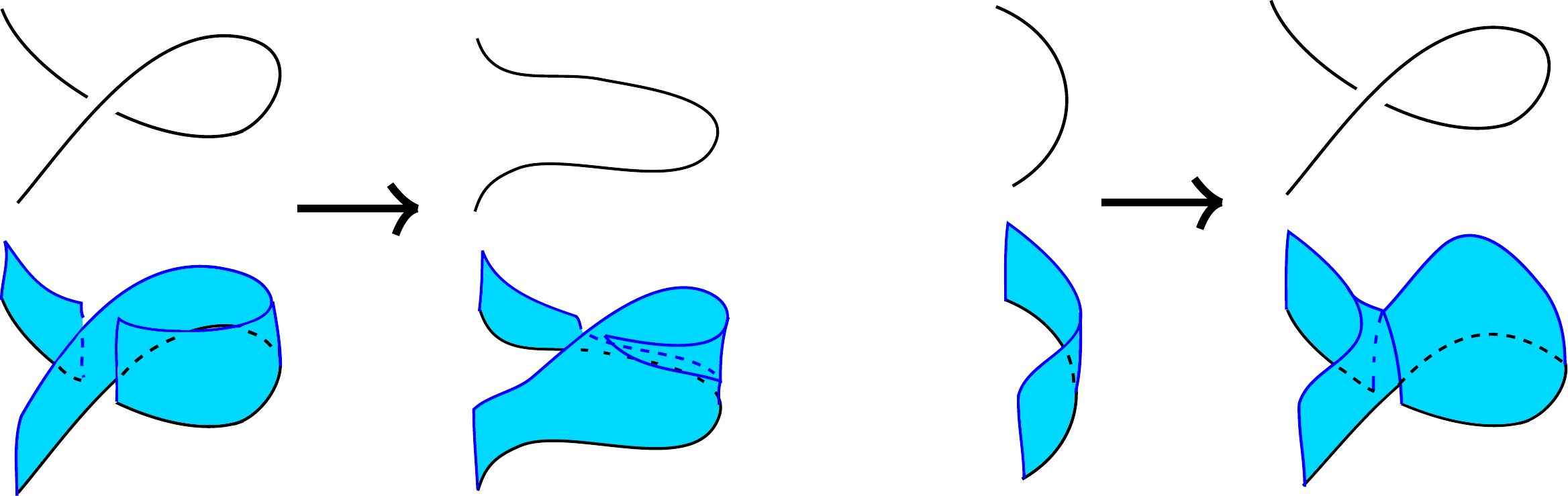}\\
\includegraphics[width=95mm]{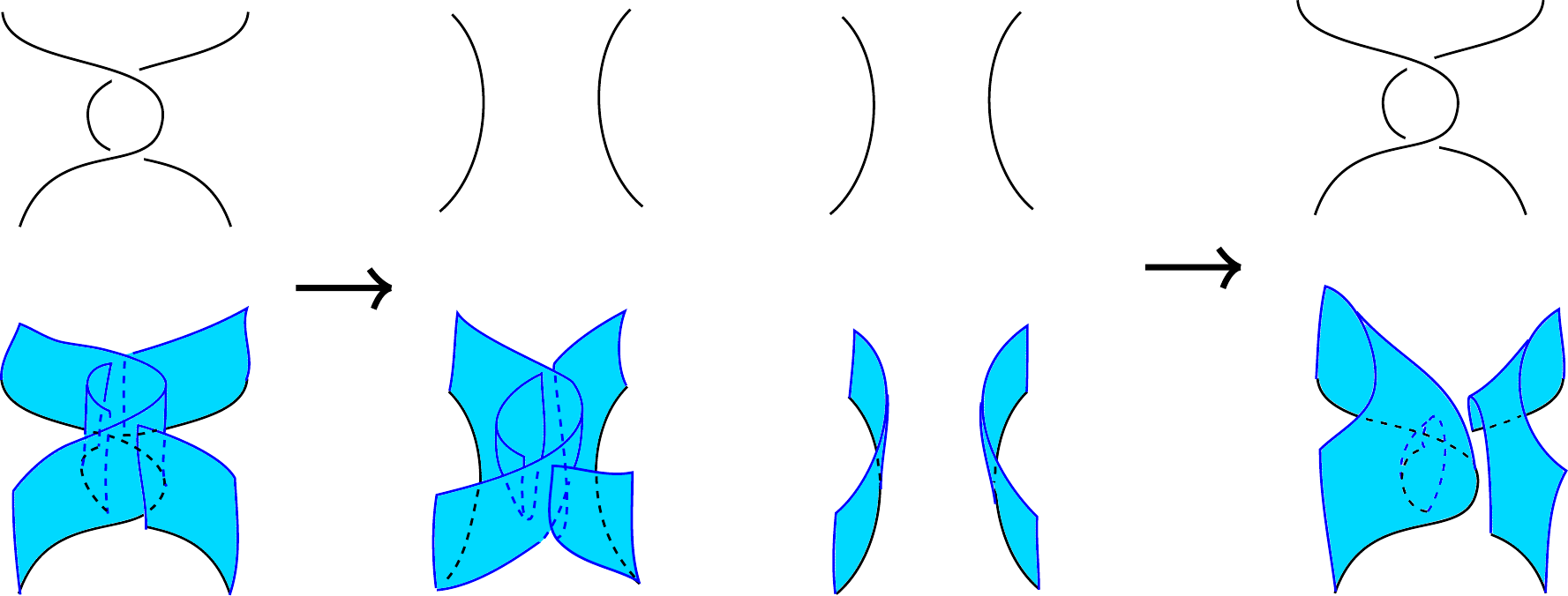}\\
\includegraphics[width=45mm]{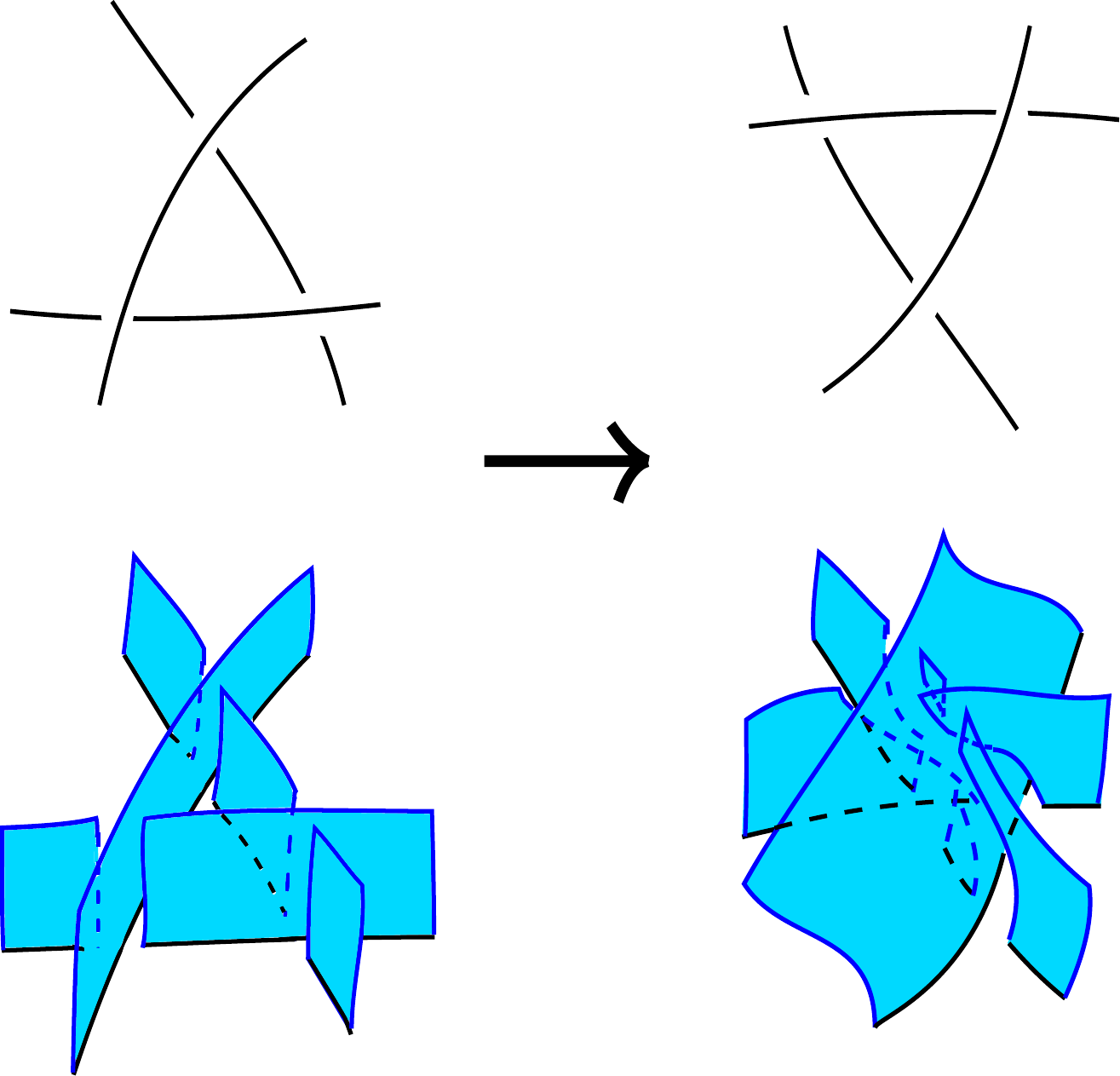}
\caption{The trace $\mathbb{T}_R$ of a Reidmeister move $R$. Below each possible Reidemeister move, we draw (on the left) a product broken surface and (on the right) the trace of the Reidemeister move. Note that changing the type of a crossing in $R$ corresponds to changing which sheet of the intersection in $\mathbb{S}$ is broken. {\bf{Top row:}} $R$ is Reidemeister~I. Note $\mathbb{S}$ contains a branch point. {\bf{Middle row:}} Reidemeister~II. {\bf{Bottom row:}} Reidemeister~III. Note $\mathbb{S}$ contains a triple point.}\label{fig:Rmoves}
\end{figure}

\begin{remark}\label{rem:sign}
If $R$ is a Reidemeister I move, $\mathbb{T}_R$ contains exactly one branch point and no triple points. The sign of the branch point depends on the sign of $R$: If $R$ is positive, i.e. the move $R$ adds a positive crossing or cancels a negative crossing, then the branch point will be negative (and vice versa). If $R$ is a Reidemeister II move, then $\mathbb{T}_R$ includes only double arcs of self-intersection. If $R$ is a Reidemeister III move, then $\mathbb{T}_R$ contains no branch points and exactly one triple point.

Similarly, if $\Rr = (R_1,\ldots, R_m)$ is a sequence of Reidemeister moves including $p$ positive RI moves, $n$ negative RI moves, $k$ RII moves and $m-(p+n+k)$ RIII moves, then the trace $\mathbb{T}_\Rr$ contains exactly $n$ positive branch points, $p$ negative branch points, and $m-(p+n+k)$ triple points.
\end{remark}

\begin{proof}[Proof of Theorem \ref{thm:getbroken}]
Recall that $\DD_i\cup\overline{\DD}_{i+1}$ is an unlink diagram. Therefore, there exists a sequence of Reidemeister moves $\Rr_i = (R_1,\ldots, R_{m_i})$ taking $\DD_i\cup\overline{\DD}_{i+1}$ to a crossingless diagram $\DD'_i$ of an unlink. Let $\mathbb{T}_i$ be the trace of $\Rr_i$. Cap off the $\DD'_i$ boundary of $\DD_i$ with trivial disks to obtain a broken surface diagram $\mathbb{S}_i$ with boundary $\DD_i\cup\overline{\DD'}_{i+1}$. 

Now embed the tri-plane in $\mathbb{R}^3$, so the diagrams $\DD_1,\DD_2,\DD_3$ lie in half-planes $P_1,P_2,P_3$ at angles $0,2\pi/3,4\pi/3$ about the $x$ axis. Note that $\DD_i$ is truly a diagram contained in a plane and not a tangle in space. Now $P_i\cup \overline{P_{i+1}}$ contains the diagram $\DD_i\cup\overline\DD_{i+1}$, which is the boundary of the broken surface $\mathbb{S}_i$ with $P_i\cup \overline{P_i+1}$. Thus, we may glue copies of $\mathbb{S}_1,\mathbb{S}_2,\mathbb{S}_3$ (correspondingly between $P_1,P_2$; $P_2,P_3$; $P_3,P_1$) to obtain a broken surface diagram $\mathbb{S}$ for $\Ss$.
\end{proof}

\subsection{Euler number and the Whitney--Massey theorem}

We can make use of Theorem~\ref{thm:getbroken} for computations that can be done with broken surface diagrams.

\begin{proposition}[\cite{EulerNumber}]\label{prop:branchsign}

Let $\mathbb{S}$ be a broken surface diagram of a surface $\Ss$. Assume $\mathbb{S}$ has $p$ positive branch points and $n$ negative branch points. Then $e(\Ss)=p-n$.

\end{proposition}

We sketch the proof of Proposition \ref{prop:branchsign} in at least as much detail as to convince the familiar reader that $e(\Ss)$ is $p-n$, rather than $n-p$.
\begin{proof}[Sketch]
Push $\Ss$ off itself and project the resulting parallel surface $\Ss'$ to $S^3$. The intersections between $\Ss$ and $\Ss'$ manifest in the projection near branch points of $\mathbb{S}$, as illustrated in Figure \ref{fig:branchsign}.
\end{proof}

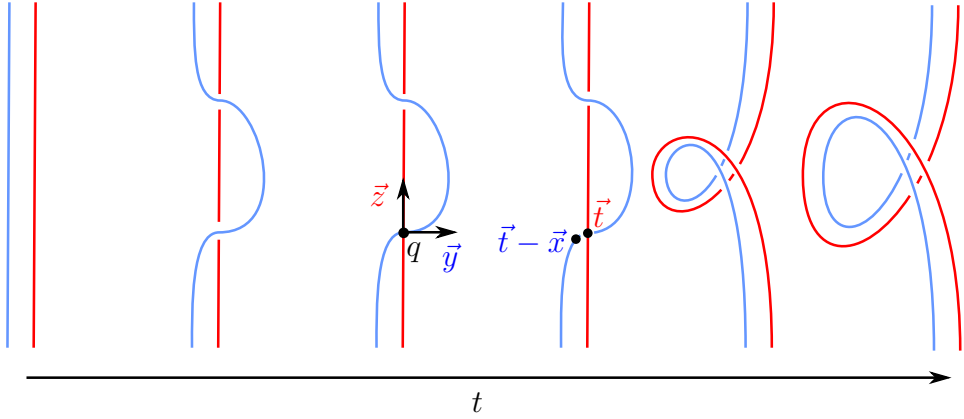
\begin{figure}{\centering
\labellist
\pinlabel {$q$} at 225 73
\pinlabel {\textcolor{red}{$\vec{z}$}} at 205 105
\pinlabel {\textcolor{blue}{$\vec{y}$}} at 245 70
\pinlabel {\textcolor{red}{$\vec{t}$}} at 327 95
\pinlabel {\textcolor{blue}{$\vec{t}-\vec{x}$}} at 290 80
\pinlabel {$t$} at 260 -10
\endlabellist
\includegraphics[width=5in]{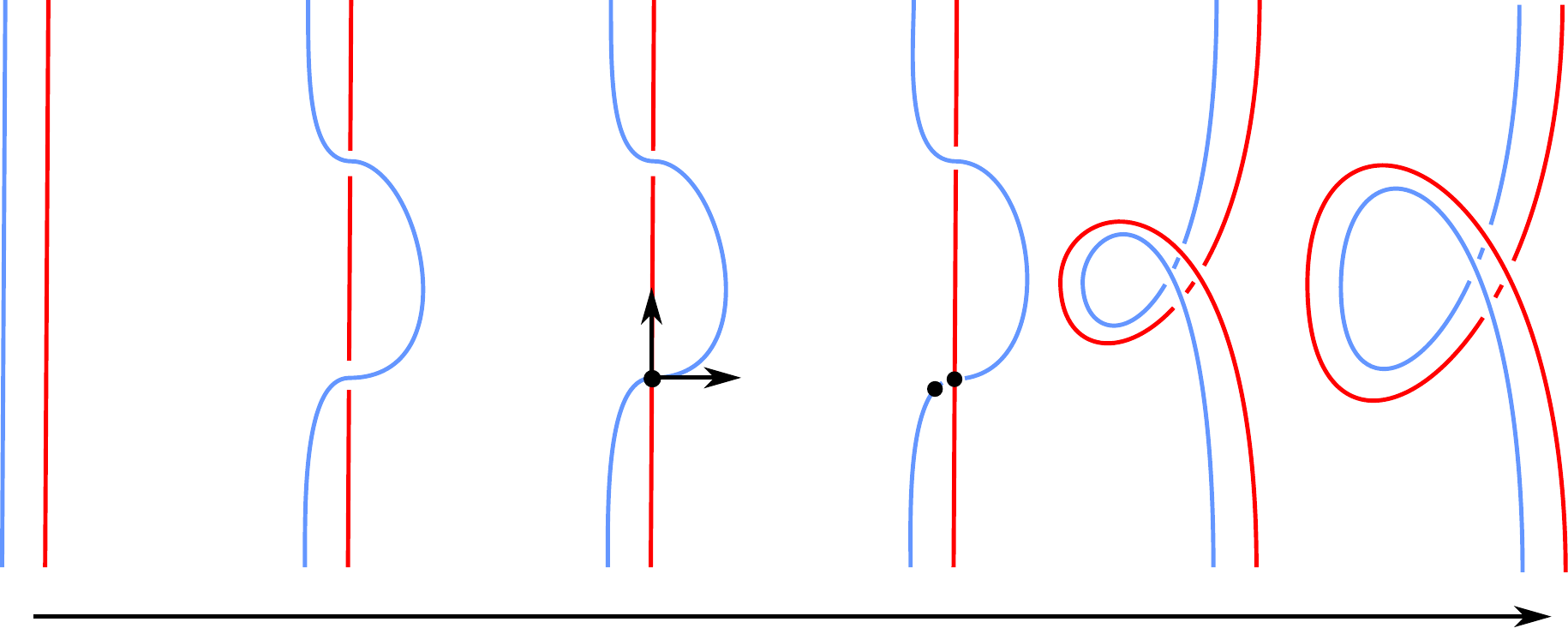}
\vspace{.2in}
\caption{A surface $\Ss$ and a parallel copy $\Ss'$ near a portion of $\Ss$ that projects to a positive branch point in a broken surface diagram. There is one intersection visible intersection point $q$ between $\Ss$ and $\Ss'$. We indicate $q$ as well as positive bases for $T_q(\Ss)$ and $T_q(\Ss')$, locally orienting $\Ss$ and inducing a local orientation on $\Ss'$. We find $T_q(\Ss)$ has positive basis $\{\vec{z},\vec{t}\}$ while $T_q(\Ss')$ has positive basis $\{\vec{y},\vec{t}-\vec{x}\}$. Since $\{\vec{z},\vec{t},\vec{y},\vec{t}-\vec{x}\}$ is a positive basis for $\R^4=\langle \vec{x},\vec{y},\vec{z},\vec{t}\rangle$, we see that $q$ is a point of positive intersection. Similarly, near a portion of $\Ss$ projecting to a negative branch point, we would find a negative intersection between $\Ss$ and $\Ss'$.
\label{fig:branchsign}
}}
\end{figure}

\begin{corollary}\label{cor:euler}
Let $\DD=(\DD_1,\DD_2,\DD_3)$ be a tri-plane diagram of a surface $\Ss\subset S^4$. Let $w_i$ be the writhe of the diagram $\DD_i\cup\overline{\DD}_{i+1}$. Then $e(\Ss)=w_1+w_2+w_3$.
\end{corollary}

\begin{proof}
Let $\mathcal{R}_i$ denote a sequence of Reidemeister moves taking $\DD_i\cup\overline{\DD}_{i+1}$ to a zero-crossing diagram. Suppose $\mathcal{R}_i$ includes $p_i$ positive RI moves and $n_i$ negative RI moves. Since RII and RIII moves preserve writhe and a zero-crossing diagram has writhe zero, we must have $n_i-p_i=w_i$.

Let $\mathbb{S}$ be the broken surface diagram obtained from $\mathcal{R}_1,\mathcal{R}_2,\mathcal{R}_3$ as in the proof of Theorem \ref{thm:getbroken}. By Remark \ref{rem:sign}, within $X_i$, $\mathbb{S}\cap X_i=\widehat{\mathbb{S}}_i$ contains $n_i$ positive branch points and $p_i$ negative branch points. Moreover, there are no branch points of $\mathbb{S}$ in $X_i\cap X_j$ for any $i\neq j$. We conclude \[e(\Ss)=(n_1+n_2+n_3)-(p_1+p_2+p_3)=(n_1-p_1)+(n_2-p_2)+(n_3-p_3)=w_1+w_2+w_3.\]
\end{proof}

From Corollary \ref{cor:euler}, it is easy to conclude that orientable surfaces in $S^4$ have trivial normal bundle. This gives an alternative argument to the usual one (that oriented surfaces have zero self-intersection number since $H_2(S^4;\mathbb{Z})=0$).  Note that $\Ss$ is oriented if and only if the bridge points and arcs of any tri-plane diagram $\DD$ are coherently oriented; see Lemma 2.1 of~\cite{MTZ_graph}.
\begin{corollary}\label{cor:orient0}
Let $\Ss$ be an oriented surface in $S^4$. Then $e(\Ss)=0$.
\end{corollary}
\begin{proof}
Let $(\DD_1,\DD_2,\DD_3)$ be an oriented tri-plane diagram for $\Ss$. Since the $\DD_i$ are oriented, we have $w(\DD_{i}\cup\overline\DD_{i+1})=w(\DD_i)-w(\DD_{i+1})$, where $\DD_i\cup\overline\DD_{i+1}$. By Corollary \ref{cor:euler}, 
\begin{align*}
 e(\Ss)&=w(\DD_1\cup\overline\DD_2)+w(\DD_2\cup\overline\DD_3)+w(\DD_3\cup\overline\DD_1)\\
&=(w(\DD_1)-w(\DD_2))+(w(\DD_2)-w(\DD_3))+(w(\DD_3)-w(\DD_1))\\
&=0.
\end{align*}
\end{proof}

It is clear that a nonorientable surface in $S^4$ has even self-intersection number, since $H_2(S^4;\mathbb{Z}/2\mathbb{Z})=0$. But with a little more work, using the above argument one can also compute the Euler number of a nonorientable surface mod 4. This corollary is sometimes called \emph{Whitney congruence}. The following corollary was originally proved by Massey~\cite{massey}.

\begin{corollary}\label{cor:nonorientable}
Let $\Ss$ be a surface in $S^4$. Then $e(\Ss)\equiv2\chi(\Ss)\pmod{4}$.
\end{corollary}

\begin{proof}[Proof of Corollary \ref{cor:nonorientable}]

The two unknotted $\mathbb{RP}^2$s have Euler numbers $+2$ and $-2$. Therefore, if $\Ss\cong\#_k\mathbb{RP}^2$ is an unknotted surface in $S^4$, $e(\Ss)=2(a-b)$ for some $a,b\in\{0,\ldots,k\}$ with $a+b=k$. Therefore, the corollary is true for unknotted surfaces.


Consider the effect of a crossing change at crossing $c$ in $\DD_i$ on $w_1$, $w_2$, and $w_3$. Since $w_{i+1}$ is the writhe of $\DD_{i+1}\cup\overline{\DD}_{i-1}$, $w_{i+1}$ remains constant. However, each of $w_i$ and $w_{i-1}$ change by $+2$ or $-2$, with sign depending on the sign of $c$ in $\DD_i\cup\overline\DD_{i+1}$ and $\DD_{i-1}\cup\overline\DD_{i}$. If $\Ss$ is not orientable, then $c$ may have the same or opposite signs in these two link diagrams. Therefore, the crossing change may preserve $w_1+w_2+w_3$, or increase or decrease the total by four. We conclude that $(w_1+w_2+w_3)\pmod{4}$ is preserved by the crossing change.

Now by \cite[Corollary~1.2]{MTZ_graph}, there exists a sequence of crossing changes transforming the triple $(\DD_1,\DD_2,\DD_3)$ into a tri-plane diagram for an unknotted surface $\Ss'$. Since $\Ss'$ is unknotted, $e(\Ss')\equiv 2\chi(\Ss')\pmod{4}$. By Corollary \ref{cor:euler}, we conclude that
$$e(\Ss)=w_1+w_2+w_3 \equiv e(\Ss')\equiv 2\chi(\Ss')=2\chi(\Ss)\pmod{4}.$$
\end{proof}

Finally, we refine Corollary \ref{cor:nonorientable} to the more general Whitney--Massey Theorem~\cite{massey}.  One of the main ingredients is the following Theorem of Viro.

\begin{theorem}\label{thm:viro}\cite{viro}
If $\Ss$ is a surface embedded in $S^4$ and $X^{\Ss}$ is the two-fold cover of $S^4$ branched along $\Ss$, then
\[ -e(\Ss) = 2 \sigma(X^{\Ss}).\]
\end{theorem}

We can now proceed with the proof, which also makes use of work by Gordon and Litherland~\cite{gordonlitherland}.

\begin{theorem}\label{thm:massey}
Let $\Ss$ be a closed, connected, non-orientable surface in $S^4$, and set $\chi:=\chi(\Ss)$. Then the Euler number $e(\Ss)$ of $\Ss$ is in the set \[\{2\chi-4,2\chi,2\chi+4,\ldots,-2\chi-4,-2\chi,-2\chi+4\}.\]

\end{theorem}

\begin{proof}
Using Corollary \ref{cor:nonorientable}, we need only prove that $|e(\Ss)|\le 4-2\chi$. Let $\DD=(\DD_1,\DD_2,\DD_3)$ be a tri-plane diagram for $\Ss$. Let $\widehat\DD_i = \DD_i\cup\overline\DD_{i+1}$.

Let $X^S$ denote the 2--fold cover of $S^4$ branched along $\Ss$. The genus zero trisection of $S^4$ lifts to a trisection $\mathbb{T}=(X^S_1,X^S_2,X^S_3)$ of $X^S$, with $X^S_i$ covering $X_i$. Let $H^S_i=X^S_i\cap X^S_{i+1}$. By Wall \cite{wall}, \[\sigma(X^S)=\Sigma_i\sigma(X^S_i)+\sigma(\nu(H^S_1\cup H^S_2\cup H^S_3)).\] Each $X^S_i$ is a 4--dimensional 1--handlebody, so has vanishing signature. Therefore, 
\[\sigma(X^S)=\sigma(\nu(H^S_1\cup H^S_2\cup H^S_3)).\]

Now fix checkerboard surfaces $F_1$, $F_2$, and $F_3$ for $\widehat\DD_1 = \DD_1\cup\overline\DD_2$, $\widehat\DD_2 = \DD_2\cup\overline\DD_3$, and $\widehat\DD_3 = \DD_3\cup\overline\DD_1$ (respectively) so that the surfaces $F_i$ and $F_{i+1}$ agree in $\DD_{i+1}$.  See Figure \ref{fig:checkerboard}.

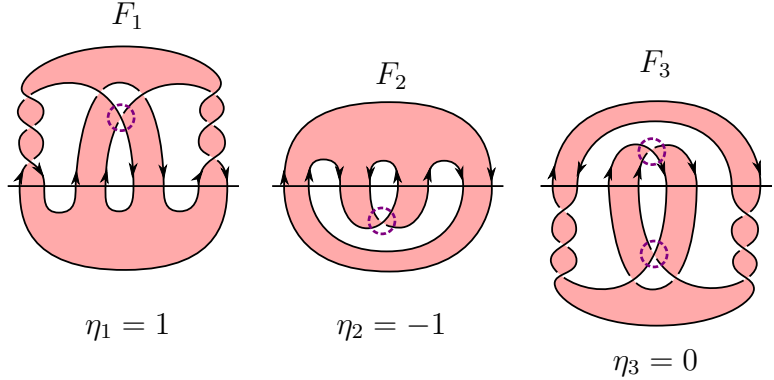
\begin{figure}{\centering
\labellist
\pinlabel {$F_1$} at 50 130
\pinlabel {$F_2$} at 160 105
\pinlabel {$F_3$} at 270 110
\pinlabel {$\eta_1=1$} at 50 0
\pinlabel {$\eta_2=-1$} at 160 0
\pinlabel {$\eta_3=0$} at 270 -15
\endlabellist
\vspace{.2in}
\includegraphics[width=4in]{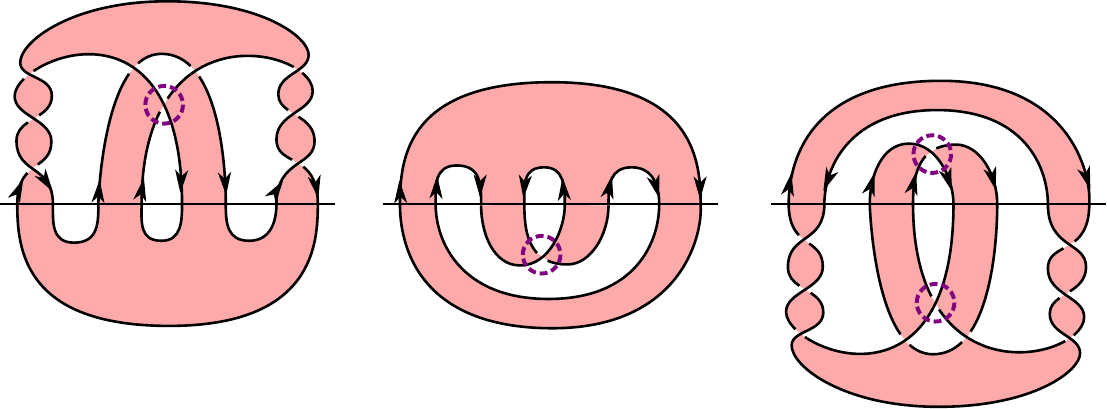}
\vspace{.2in}
\caption{Checkerboard surfaces $F_1,F_2,F_3$ for $\widehat{\DD}_1,\widehat{\DD}_2,\widehat{\DD}_3$, respectively. We choose the surfaces so that $F_i$ and $F_{i+1}$ agree in $\DD_{i+1}$. We arbitrarily choose some orientations of each $\widehat{\DD}_i$ (indicated by arrows) and then use dashed circles to indicate the type II crossings (see Figure \ref{fig:type12}) of $F_i$ given these orientations.}
\label{fig:checkerboard}
}
\end{figure}

Let $\Ss'$ be a surface obtained by gluing together $F_1,F_2,F_3$ along common boundary, after pushing the interior of $F_i$ slightly into $X_i$. 

\begin{claim}\label{claimes0}The surface $\Ss'$ is unknotted with $e(\Ss')=0$.
\end{claim}
\begin{proof}
Let $F'_i$ denote the copy of $F_i$ pushed into $B^4$, so $\Ss'=F'_1\cup F'_2\cup F'_3$.
Let $H$ be the 3--manifold formed as the union of the traces of the three isotopies pushing the $F_i$ into $B^4$.
Then, $H$ is a 3--dimensional neighborhood of a union of three 1--dimensional spines of the $F_i$ (that are chosen to agree at $\Sigma$).
In other words, $H$ is a handlebody, though it may be non-orientable.
In any event, $\Ss'$ is unknotted with $e(\Ss')=0$, since it bounds a handlebody in $S^4$.
\end{proof}

Let $X^F_i$ denote the 2--fold covering of $X_i$ branched along $F'_i$. Let $G_i$ denote the Gordon-Litherland form associated to $F_i$~\cite{gordonlitherland}.

\begin{claim}\label{claimsumgoeritz}
We have \[\sigma(X^F_i)=\sigma(G_i).\]
\end{claim}
\begin{proof}
Gordon--Litherland \cite{gordonlitherland} showed that the quantity $\sigma(X^F_i)+e(F'_i)/2$ is independent of the choice of checkerboard surface $F_i$, up to Reidemeister moves of the oriented diagram $\widehat{\DD}_i$.
Since $\widehat{\DD}_i$ is a diagram of an unlink, we conclude that $\sigma(X^F_i)+e(F'_i)/2=0$. By Gordon--Litherland, we also have $\sigma(G_i)+e(F_i)/2=0$, yielding the desired equality.
\end{proof}

\begin{claim}\label{claimes}
We have
$$e(\Ss)=
2(\sigma(G_1)+\sigma(G_2)+\sigma(G_3)).$$
\end{claim}
\begin{proof}
We remind the reader of the following theorem of Gordon--Litherland \cite{gordonlitherland}: if $G$ is a Goeritz matrix for a diagram of a link $L$ associated to a checkerboard surface $F$, then $\sigma(L)=\sigma(G)-\eta$, where $\eta$ is a sum of signs over type II crossings in $F$ (see Figure \ref{fig:type12}). Since each $\widehat{\DD}_i$ is a diagram for an unlink (which has signature zero), we conclude $\sigma(G_i)=\eta_i$, where $\eta_i$ is the corresponding sum of signs over type II crossings in $F_i$.
\begin{observation*}
A crossing $c$ in $\DD_i$ has the same sign in $\widehat{\DD}_i$ as it does in $\widehat{\DD}_{i-1}$ if and only if it is type I in one of $F_i$ or $F_{i-1}$ and type II in the other.\end{observation*}
If $c$ has different signs in $\widehat{\DD}_i$ and $\widehat{\DD}_{i-1}$, then it does not contribute to $e(\Ss)=\Sigma_i w(\widehat{\DD}_i)$. 
If $c$ has the same sign in each of $\widehat{\DD}_i,\widehat{\DD}_{i-1}$, then $c$ contributes twice that sign and is type II in exactly one of $F_i,F_{i-1}$ by the above observation. We conclude the claim.
\end{proof}

\begin{figure}{\centering
\labellist
\pinlabel {Type I} at 15 -10
\pinlabel {Type II} at 132 -10
\endlabellist
\includegraphics[width=2in]{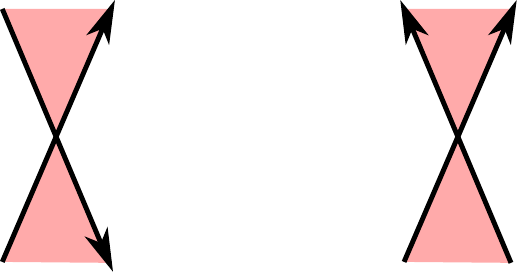}
\vspace{.2in}
\caption{A crossing $c$ in an oriented link diagram $\widehat\DD$. The shaded region indicates a checkerboard surface $F$ for $\widehat\DD$. We say $c$ is type I if $F$ can be locally oriented near $c$ to induce the correct orientation on $\partial F$. Otherwise, $c$ is type II. In this picture, it does not matter which strand contains the overcrossing; type is independent of sign. }
\label{fig:type12}
}
\end{figure}

Let $X^{\Ss'}$ be the 2--fold cover of $S^4$ branched along $\Ss'$. The splitting~$S^4=X_1\cup X_2\cup X_3$ lifts to a splitting (not a trisection) $X^{\Ss'}=X^{F}_1\cup X^{F}_2\cup X^{F}_3$. Let $H^{\Ss'}_i=X^{F}_i\cap X^{F}_{i+1}$. Again by Wall \cite{wall}, we have \[\sigma(X^{\Ss'})=\Sigma_i\sigma(X^{F}_i)+\sigma(\nu(H^{\Ss'}_1\cup H^{\Ss'}_2\cup H^{\Ss'}_3)).\] 
By Claim \ref{claimsumgoeritz}, $\sigma(X^F_i)=\sigma(G_i)$. 
Moreover, note that
$$\nu(H^{\Ss'}_1\cup H^{\Ss'}_2\cup H^{\Ss'}_3)\cong \nu(H^{S}_1\cup H^{S}_2\cup H^{S}_3).$$
We conclude 
$$\sigma(X^{\Ss'})=\sigma(X^S)+\Sigma_i\sigma(G_i).$$

By Claim \ref{claimes}, $\Sigma_i\sigma_i(G_i)=e(\Ss)/2$. Moreover, since $\Ss'$ is an unknotted surface with $e(\Ss')=0$ (Claim \ref{claimes0}), $X^{\Ss'}\cong \#_n\CP^2\#_n\overline{\CP}^2$ for some $n\ge 0$. Therefore, $\sigma(X^{\Ss'})=0$, so this becomes
$$e(\Ss)=-2\sigma(X^S).$$

Finally, we have $|\sigma(X^S)|\le b_2(X^S)=2-\chi$. Thus, we obtain our desired inequality: \[|e(\Ss)|\le 4-2\chi.\]

\end{proof}

\subsection{The triple point number of a bridge trisection}\label{subsec:triple_point}

Recall from Remark \ref{rem:sign} that given a tri-plane diagram $(\DD_1,\DD_2,\DD_3)$ of a surface $\Ss$, we may produce a broken surface diagram of $\Ss$ with $t_1+t_2+t_3$ triple points, where $t_i$ is the number of RIII moves in some sequence of Reidemeister moves transforming $\DD_i\cup\overline\DD_{i+1}$ into a crossingless diagram. This allows us to define the triple point number of a bridge trisection as follows.

\begin{definition}
Let $\widehat\DD$ be an unlink diagram. We say a sequence of Reidemeister moves applied to $\widehat\DD$ is an \emph{uncrossing sequence for $\widehat\DD$} if the end result is a crossingless diagram. We define
$$t(\widehat\DD)=\text{the minimum number of RIII moves in any uncrossing sequence for $\widehat\DD$.}$$

\end{definition}
\begin{definition}
Let $\DD=(\DD_1,\DD_2,\DD_3)$ be a tri-plane diagram of a bridge trisection $\TT$ of a knotted surface $\Ss$. Define $t(\DD)=t(\widehat\DD_1)+t(\widehat\DD_2)+t(\widehat\DD_3)$, and define $t(\TT)$ to be the minimal value of $t(\DD)$, taken over all tri-plane diagrams $\DD$ of $\TT$. This is called the {\em triple point number} of $\TT$.
\end{definition}

By construction, this triple point number is an invariant of the bridge trisection. By Remark \ref{rem:sign}, we have $t(\TT)\geq t(\Ss)$, where $t(\Ss)$ is the usual triple point number of the surface $\Ss$ (i.e. the minimum number of triple points in any broken surface diagram of $\Ss$) for any bridge trisection $\TT$ of a surface $\Ss$.

\begin{questions}\leavevmode
\begin{enumerate}
\item Given a surface $\Ss$, is there a bridge trisection $\TT$ for $\Ss$ with $t(\TT)=t(\Ss)$? 
\item Does there exist a surface $\Ss$ with bridge trisection $\TT$ so that $t(\TT)>t(\Ss)$?
\item Does there exist a bridge trisection $\TT$ with $\sS$ an unknotted 2-sphere so that $t(\TT)>0$?
\end{enumerate}
\end{questions}

By construction, ribbon surfaces (defined below) always have triple point number zero. In the next subsection, we show that every ribbon surface has a ribbon bridge trisection, and that ribbon bridge trisections always have triple point number zero, thus recovering this fact. 

\subsection{Ribbon bridge trisections}\label{subsec:ribbon}

In this subsection we define bridge trisections for ribbon surfaces arising naturally from ribbon presentations. In Subsection~\ref{subsec:Nielsen} we will use this analysis to give examples of ribbon 2--knots that admit non-isotopic minimal bridge trisections.

One of the simplest classes of knotted surfaces is that of {\em ribbon surfaces}, which bound embedded handlebodies in $B^5$ with only index 0 and 1 critical points with respect to the radial height function. Equivalently, an oriented surface in $S^4$ is ribbon if it bounds a ribbon-immersed handlebody in $S^4$. Ribbon surfaces can also be described by ribbon presentations. 

\begin{definition}
Let $L=L_1\cup\cdots\cup L_n$ be an unlink of oriented 2-spheres in $S^4$. For some $m\ge n-1$, let $H=\{h_1,\ldots, h_m\}$ be disjoint embeddings of 3-dimensional 1-handles $I\times D^2$ in $S^4$ such that for each~$i$,
\begin{itemize}
\item Each $h_i(I\times D^2)$ meets $L$ exactly in its attaching region $h_i(\partial I\times D^2)$, and is not tangent to $L$ near this attaching region.
\item $(L\setminus H) \bigcup_{i=1}^m h_i(I\times \partial D^2)$ is a connected, oriented surface $\sS$ (of genus $m-n+1$).
\end{itemize}
The data $(L,H)$ is a {\emph{ribbon presentation for $\sS$.}}

\end{definition}

\begin{figure}
  \centering
\includegraphics[width=80mm]{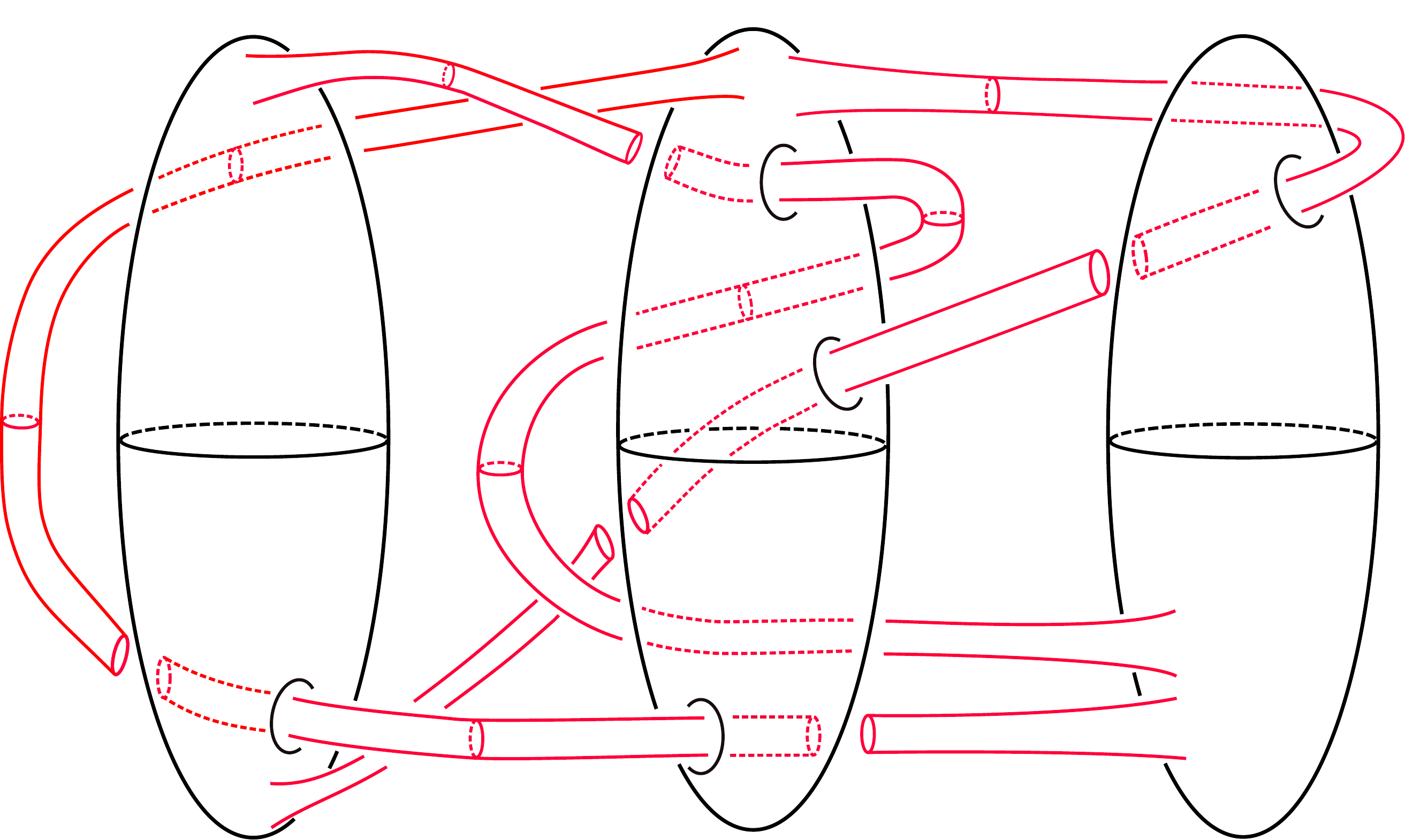}
  \caption{A broken surface diagram indicating a ribbon presentation for a knotted torus.}
  \label{fig:ribbon_pres}
\end{figure}

In short, a ribbon presentation is a description of a surface obtained by fusing an oriented unlink together along oriented tubes. A ribbon presentation has an especially nice broken surface diagram, where the only intersections are double circles between tubes and spheres (See Figure \ref{fig:ribbon_pres}). 


The {\em tube map} encodes a broken surface diagram of a ribbon surface with a virtual graph. Yajima defined the tube map as a diagrammatic operation from classical knots (resp. arcs) to ribbon tori (resp. spheres) \cite{yajima}. Satoh extended the tube map to include virtual crossings, and proved that it is surjective onto ribbon spheres and tori \cite{satoh}. Finally, Kauffman and Martins defined the notion of a virtual graph, allowing for higher genus surfaces \cite{kauffman}. 

In Figure \ref{fig:ribbon_tri-plane}, we illustrate in the first two frames the procedure for obtaining a banded unlink diagram of $\tube(G)$ from $G$. When two edges in $G$ have a virtual crossing, the apparent ``crossing" of the tubed surface may be chosen arbitrarily (the two choices yield isotopic surfaces in $S^4$). The orientations of the overstrands of $G$ determine the crossings of the banded unlink diagram near any classical crossing of $G$ (cf.\ Figure \ref{fig:ribbon_pres}).

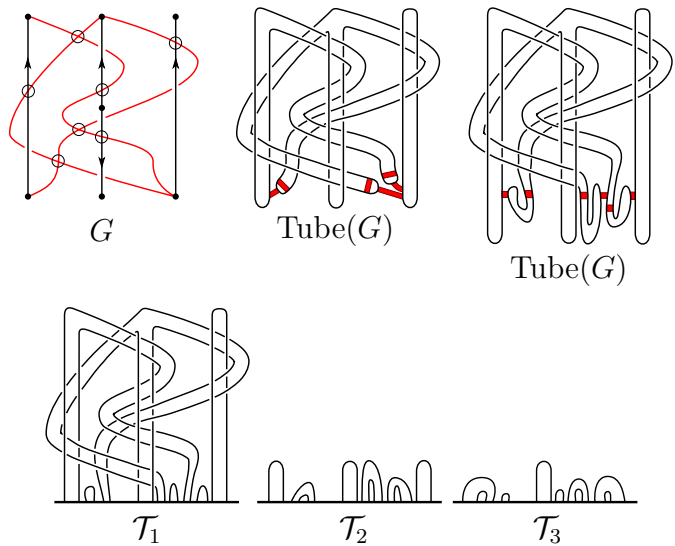
\begin{figure}
  \centering
  \labellist
   \pinlabel {$G$} at 35 103
   \pinlabel {$\tube(G)$} at 123 103
   \pinlabel {$\tube(G)$} at 211 88
  \pinlabel {$\mathcal{T}_1$} at 52 -10
  \pinlabel {$\mathcal{T}_2$} at 130 -10
  \pinlabel {$\mathcal{T}_3$} at 203 -10
  \endlabellist
 \includegraphics{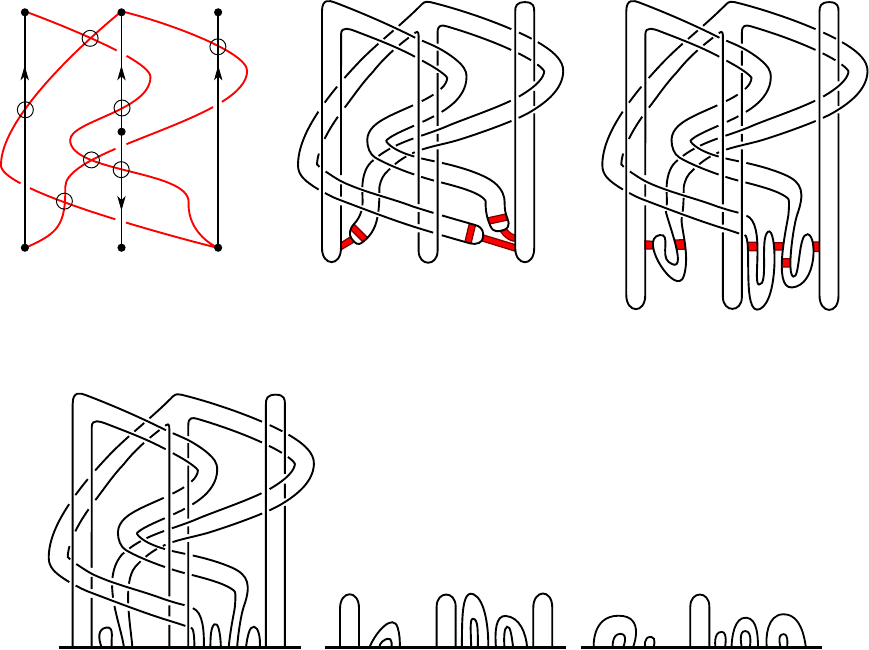}
 \vspace{.15in}
  \caption{Top left: a virtual graph $G$ in 3-bridge position corresponding to the ribbon presentation in Figure \ref{fig:ribbon_pres}. Top middle: A banded unlink diagram for $\tube(G)$. Top right: We perturb the banded unlink diagram to be in 9-bridge position. Bottom: The resulting $(9;3)$--tri-plane diagram of $\tube(G)$.}
  \label{fig:ribbon_tri-plane}
\end{figure}

Via the tube map, a virtual graph can be thought of as a shorthand for a ribbon presentation, where overstrands become spheres in the ribbon presentation and understrands joining them become tubes. A virtual graph diagram is in $n$-\emph{bridge position} if, considered as an immersed graph in $\R^2$, the height function on the graph is Morse, and has $n$ minima and $n$ maxima. Now we show how a virtual graph in bridge position gives rise to a bridge trisection whose parameters are determined by the bridge index and Euler characteristic of the graph.

\begin{proposition}\label{prop:ribbonprestograph}
Let $(L,H)$ be a ribbon presentation with $n$ spheres and $m$ tubes for a surface $\sS$ of genus $g=m-n+1$. Then there is a virtual graph $G$ such that: 
\begin{enumerate}
\item $Tube(G)=\sS$
\item $G$ has Euler characteristic $\chi(G)=1-g=n-m$ 
\item $G$ can be put into $n$-bridge position.
\end{enumerate}
\end{proposition}

\begin{proof}
There is an obvious broken surface diagram of $\sS$ which `comes from' the ribbon presentation, i.e.\ the unlink $L$ is projected into $\R^3$ so that it is embedded and so that the components of $L_i$ bound disjoint 3-balls in $\R^3$. The projections of the 3-dimensional 1-handles $h_i$ are embedded in $\R^3\setminus L$, and only intersect the 2-spheres in the attaching region $h_i(\partial I\times D^2)$ and a finite number of disks $h_i(\{t\}\times D^2)$. The boundaries of these disks are double point circles, and they are the only self-intersections of the projection of $\sS$. As mentioned above, we can arrange that a tube never crosses the same sphere $L_i$ over or under twice in a row. 
As we traverse the $I$ direction of a tube, it goes through double circle crossings $c_{11},c_{12};c_{21},c_{22};\dots;c_{k1},c_{k2}$, where $c_{i1},c_{i2}$ are crossings with the same component $L_j$, and have opposite over/under information. See Figure \ref{fig:ribbon_pres}.

Now, we construct the graph $G$ in $n$-bridge position:
first, we draw $n$ vertical edges in $\R^2$ for the $n$ components of $L$, with vertices at heights 0 and 1. Let $L_i$ and $L_j$ be the components of $L$ that the first tube $h_1(I\times\partial D^2)$ attaches to. We draw an edge of $G$ joining the bottom endpoint of $L_i$ to the top endpoint of $L_j$, traversing monotonically upwards. For each pair of crossings $c_{i1},c_{i2}$ of a tube with a sphere $L_j$, the edge corresponding to the tube crosses under the vertical edge representing $L_j$. We remember the sign of the crossing by a local orientation of the overstrand: the conormal (in $\R^2$) to the overstrand points to the `under' double circle crossing, as in Figure \ref{fig:ribbon_tri-plane}. We continue in this way, adding an edge for each tube in $H$. When an edge needs to get to the other side of another edge without crossing, a virtual crossing is used. The graph $G$ produced has $2n$ vertices and $n+m$ edges, so its Euler characteristic is $n-m$. The tube of this graph is the same broken surface diagram we began with, so $Tube(G)=\sS$. By construction, $G$ is in $n$-bridge position.
\end{proof}

\begin{proposition}\label{prop:ribbontrisection}
Suppose $\sS$ is a ribbon surface admitting a ribbon presentation $(L,H)$ consisting of $n$ spheres and $m$ tubes. Then $\sS$ admits an $(n+2m;n)$-bridge trisection.
\end{proposition}

\begin{proof}

Given $(L,H)$, first construct a virtual graph $G$ in $n$-bridge position as in Proposition \ref{prop:ribbonprestograph}. In Figure \ref{fig:ribbon_tri-plane}, we illustrate how to obtain a bridge trisection of $\sS$ from $G$. We first obtain a banded unlink diagram of $\sS$ in which the unlink is in $n$-bridge position and there are $2m$ bands so that surgering the unlink along the bands yields an $n$-component unlink. We perturb once near each band to obtain a banded unlink diagram in $(n+2m)$-bridge position. We thus obtain a $(b;(c_1,c_2,c_3))$-bridge trisection of $\sS$ with
\begin{align*}
b&=n+2m,\\
c_1&=n\text{\quad (the number of unlink components)}\\c_3&=n\text{\quad (the number of unlink components after band surgery)},\\
c_2&=\chi(\sS)+b-c_1-c_3\\
&=2(n-m)+(n+2m)-n-n\\
&=n.
\end{align*}

That is, we obtain an $(n+2m;n)$-bridge trisection of $\sS$, by~\cite[Lemma~3.2]{MeiZup_bridge1}.
\end{proof}

In Subsection~\ref{subsec:Nielsen}, we will show that by using the construction of Proposition \ref{prop:ribbontrisection} on distinct ribbon presentations of the same 2-knot, one can obtain distinct bridge trisections of the same surface, both with minimal parameters.

\begin{definition}
A {\it ribbon bridge trisection} is any bridge trisection obtained from the construction of Proposition \ref{prop:ribbontrisection}.
\end{definition}

Recall that $t(\TT)$ denotes the triple point number of the trisection $\TT$ (cf. Subsection~\ref{subsec:triple_point}).

\begin{proposition}
If $\TT$ is a ribbon bridge trisection, then $t(\TT)=0$.
\end{proposition}

\begin{proof}
Let $(\DD_1,\DD_2,\DD_3)$ be a ribbon bridge trisection diagram as obtained in Proposition \ref{prop:ribbontrisection}. Each unlink diagram $\DD_i\cup\overline{\DD_{i+1}}$ is either crossingless or can be made crossingless via only RII moves. Thus, $t(\TT)=0$.
\end{proof}

\section{The fundamental group, the peripheral subgroup, and quandle colorings}\label{sec:pi1}

In this section we describe a number of ways to calculate a presentation of the fundamental group of the exterior of a surface-knot from a tri-plane diagram for the surface.  We also discuss diagrammatic approaches to Fox colorings and, more generally, quandle colorings of surface-knots, and describe a way to present the peripheral subgroup of a surface-knot.  Our approaches give rise to some interesting group-theoretic questions about tri-plane diagrams.

\subsection{The fundamental group}
Applying van-Kampen's theorem to the exterior of the bridge trisection yields the following cube of pushouts. Let $\bold p$ denote the set of $2b$ intersections of $\Sigma$ with $\sS$. The three presentation types of Theorem \ref{thm:pres} correspond to choosing a group $G$ from the first, second or third column of this cube to express $\pi_1(S^4\setminus \sS)$ as a quotient of $G$.

\begin{center}
\begin{tikzcd}[row sep=scriptsize, column sep=scriptsize]
    & \pi_1(H_1\setminus \T_1) \arrow[r] \arrow[dr] &  \pi_1(X_1\setminus \mathcal{D}_1) \arrow[dr]  &\\
    \pi_1(\Sigma\setminus  \bold p) \arrow[r]\arrow[dr]\arrow[ur]  &   \pi_1(H_2\setminus \T_2) \arrow[ur, crossing over] \arrow[dr]   &   \pi_1(X_3\setminus \mathcal{D}_3) \arrow[r]  &   \pi_1(S^4\setminus \sS)\\
    &   \pi_1(H_3\setminus \T_3) \arrow[r] \arrow[ur] &   \pi_1(X_2\setminus \mathcal{D}_2) \arrow[ur] \arrow[ul, leftarrow, crossing over] &
\end{tikzcd}
\end{center}

\begin{theorem}\label{thm:pres}
\label{thm:grp_pres}
	Let $\DD$ be a $(b;\bold c)$--tri-plane diagram for a surface knot $\Ss\subset S^4$.  Then $\pi_1(S^4\setminus\nu(\Ss))$ admits a presentation of each of the following types:
	\begin{enumerate}
		\item\label{firstpres} $2b$ meridional generators and $3b$ Wirtinger relations,
		\item $b$ meridional generators and $2b$ Wirtinger relations, or
		\item\label{thirdpres} $c_i$ meridional generators and $b$ Wirtinger relations (for any $i\in\Z_3$).
	\end{enumerate}
Moreover, these presentations can be obtained explicitly from $\DD$.
\end{theorem}
\begin{proof}
These presentations can be calculated from a tri-plane diagram by carrying out the following corresponding processes. In all cases, begin by orienting each strand of each tangle. If $\Ss$ is orientable, then it will be possible (but not necessary) to orient the tangles compatibly in the sense that the three arcs adjacent at each bridge point will be all oriented away from or all oriented toward the bridge point (see Lemma 2.1 from~\cite{MTZ_graph}). The basepoint $q$ of all of these presentations lies in the bridge sphere (away from $\Ss$) so that it is \emph{above} the tri-plane onto which $\Ss$ is projected to give $\DD$. To choose curves from the basepoint about a meridian of $\Ss$ depicted in a tangle $\DD_i$ of $\DD$, we choose an arc $\eta$ in $\Ss$ from the basepoint to that meridian whose projection to $\DD_i$ has only over crossings. Note that when $\eta$ is projected to $\DD_{i+1}$ or $\DD_{i-1}$, its projection will also only have over crossings, so this choice may be made consistently.

	\begin{enumerate}
		\item Assign labels $\{x_i\}_{i=1}^{2b}$ to the $2b$ common bridge points of the tangle diagrams $\DD_i$. These labels will represent the meridional generators in our presentation. For each arc adjacent to the bridge point labeled $x_i$, extend the label over the arc as $x_i$ if the arc is oriented away from the bridge point, and extend the label over the arc as $\overline x_i$ if the arc is oriented toward the bridge point.  Now, percolate the labels throughout each tangle diagram by applying the Wirtinger algorithm at each crossing, moving up through the height gradient of each tangle diagram.  The $3b$ relations come from the equalities encountered at the $3b$ arcs containing maximum points of the tangle diagrams.
		\item  Assign labels $\{x_i\}_{i=1}^{b}$ to the $b$ arcs containing maximum points of one of the three tangle diagrams $\DD_i$.  Percolate the labels throughout the tangle diagram by applying the Wirtinger algorithm at each crossing, moving down through the height gradient of the tangle diagram.  After finding labels for the $2b$ bridge points, and equating these with the meridians to the bridge points in the other two tangle diagrams, percolate upwards in these diagrams, eventually obtaining $2b$ relations when these arcs join together at their maxima.
Here, the orientations of the arcs are important: If $w$ and $w'$ are two words labeling two arcs that meet at a bridge point, the resulting relation is $w'=w$ if the orientations of the two arcs agree (are both outward or inward) at the bridge point, and the resulting relation is $w' = \overline w$ if the orientations disagree.
		\item First, apply tri-plane moves to remove the crossings from the tangle diagrams $\DD_i$ and $\DD_{i+1}$ for some fixed $i\in\Z_3$. This is possible because $\widehat\DD_i = \DD_i\cup\overline\DD_{i+1}$ is a diagram for a $c_i$--component unlink, and unlinks admit unique bridge splittings at each level of complexity (i.e., based on the number of bridges of each component)~\cite{NegOki,Otal}. Assign labels $\{x_i\}_{i=1}^{c_i}$ to the $c_i$ components of the unlink diagram $\DD_i$. (Here, it is best to orient the strands of $\widehat\DD_i$ coherently.) This induces labels at the $2b$ common bridge points.  Percolate the labels throughout $\DD_{i+2}$ by applying the Wirtinger algorithm at each crossing, moving up through the height gradient of the tangle diagram.  The $b$ relations come from the equalities encountered at the arcs containing the $b$ maximum points of the tangle diagram.
	\end{enumerate}

	We now describe why the processes given above work to calculate $\pi_1(S^4\setminus\nu(\Ss))$.
	Let $X = (H_1\cup H_2\cup H_3)\setminus\nu(\T_1\cup\T_2\cup\T_3)$.
	Let $q$ be a point in $\Sigma\setminus\nu\T_i$.
	It should be clear the Wirtinger algorithms outlined calculate the group $\pi_1(X,q)$.
	However, we have that $\pi_1(X,q)\cong \pi_1(S^4\setminus\nu(\Ss))$, since $S^4\setminus\nu(\Ss)$ is built from $X\times I$ by attaching only (4--dimensional) 3--handles and 4--handles.
	
\end{proof}

\begin{remarks}
\label{rmk:grp_pres}
	Presentation~(3) is strengthened in Proposition~4.5 of~\cite{MeiZup_bridge1} to a presentation with $c_i$ generators and $b-c_j$ relations, for any distinct $i,j\in\Z_3$. This is optimal from the perspective of group deficiency, and shows that the deficiency of $\pi_1(S^4\setminus\nu(\Ss))$ is at least $c_i+c_j-b$.
\end{remarks}

\subsection{The peripheral subgroup}

Once the Wirtinger algorithm has been completed, it is simple to write down the generators of a peripheral subgroup of $\Ss$ in terms of these Wirtinger generators for $\pi_1(S^4\setminus \Ss)$. The inclusion $\partial \nu \Ss \hookrightarrow S^4\setminus \nu \Ss$ induces a homomorphism $\pi_1(\partial \nu \Ss)\rightarrow \pi_1(S^4\setminus \nu \Ss)$, unique up to a choice of meridian. The image of this homomorphism is the {\emph{peripheral subgroup}} of $\Ss$, whenever $\Ss$ is connected. See \cite{kazama} for some background on the peripheral subgroups of knotted tori. If $\Ss$ has more than one component, we can still consider the image of the induced map from the boundary of a tubular neighborhood of one component of $\Ss$ into the exterior of $\Ss$.

The procedure is as follows, for connected $\Ss$.

\textbf{Step 1}:  Choose a basepoint $y$ for $\pi_1(\Ss)$ to be one of the bridge points, where a tangle arc meets the bridge sphere, call the meridian to this arc  $\mu$. There is an arc $\eta$ from the basepoint $q$ of $\pi_1(S^4\setminus\Ss)$ to $y$ lying on the bridge sphere.

\textbf{Step 2}:  Choose a generating set $\gamma_1,\dots,\gamma_n$ for $\pi_1(\Ss,y)$  so that each $\gamma_i$ is a union of tangle arcs. Write each of the generators as a word in the Wirtinger labels (traverse the curve once, starting at $y$).

\textbf{Step 3:} Push each $\gamma_i$ off $\Ss$ (using the arc $\eta$ from $y$ to $q$), then add a multiple of $\mu$ to arrange for each push-off to be nullhomologous in the complement of $\Ss$.  Push $\eta$ off with the curve, so that the curve is a based loop $\gamma'_i$ in $S^4$.

\begin{lemma}The subgroup $\langle \mu, \gamma'_1, \ldots, \gamma'_n \rangle$ of $\pi_1(S^4\setminus \nu \Ss)$ is the peripheral subgroup of $\Ss$.
\end{lemma}
\begin{proof}
This follows essentially from the definition of peripheral subgroup; note that if $x$ is pushed along $\eta$ to lie in $\partial(\nu(\Ss))$, then $\pi_1(\partial(\nu(\Ss)),x)=\langle \mu, \gamma'_1, \ldots, \gamma'_n \rangle$ .
\end{proof}

Once the generating set is established, one could use Schreier's lemma to get a presentation for the peripheral subgroup.

\begin{example}\label{ex:kinoshita}
In Figure \ref{fig:kinoshita}, we draw a tri-plane diagram of a link $\Ll=P_1\sqcup P_2$ of two unknotted projective planes; here $b=4$ and $c_i=2$ for all $i\in\Z_3$. Taken in isolation, the surfaces $P_1$ and $P_2$ are the unknotted projective planes $P_+$ and $P_-$, respectively. Since the union of the first two tangle diagrams has no crossings, we find a presentation of $\pi_1(S^4\setminus \Ll)$ as in Theorem~\ref{thm:grp_pres}(\ref{thirdpres}). We implicitly add the relations corresponding to the trivial tangles in $\DD_1$ and $\DD_2$ to see that the leftmost meridians in $\DD_3$ correspond to the same generator (up to orientation), as do the rightmost. Then we apply the Wirtinger algorithm to $\DD_3$ to obtain a relation for each of the four maxima. One relation corresponding to each of $P_1$ and $P_2$ is redundant, so we are left with the final presentation \begin{align*}\pi_1(S^4\setminus L)&=\langle a,b\mid\overline{b}aba=\overline{a}bab=1\rangle\\&=\langle a,b\mid a^2\overline{b}^{2}=\overline{a}bab=1\rangle\\&\cong Q_8\hspace{.1in}(a\mapsto i, b\mapsto j)\end{align*} 

We indicate the generator of $\pi_1(P_1)$ in bold/purple in Figure~\ref{fig:kinoshita}. Since the bold strand has a single undercrossing in the diagram, we add a canceling undercrossing to indicate a parallel copy of this curve (taking the basepoint to lie in $\boundary(\nu(P_1))$ that is nullhomologous in $S^4\setminus P_1$). This parallel copy represents $b$ in $\pi_1(S^4\setminus\Ll)$. We conclude that the peripheral subgroup of $P_1$ in $S^4\setminus\nu(L)$ is generated by the meridian $a$ and this parallel curve $b$; hence is isomorphic to $Q_8$. (By symmetry, so is the peripheral subgroup of $P_2$.)

As a consequence, since the peripheral subgroup of each $P_i$ is not $\mathbb{Z}_2$, the link $\Ll$ cannot factor as $P_{\pm}\# \Ll'$ for any link $\Ll'$ of a 2--sphere and an $\mathbb{RP}^2$. This implies that the analog of the {\emph{Kinoshita conjecture}} (that every projective plane in $S^4$ factors as the connected sum of $P_{\pm}$ and a knotted 2--sphere) is false for multiple component links. This example was first noted by Yoshikawa \cite{yoshikawa}.

In Figure~\ref{fig:kinoshita2}, we generalize $\Ll$ to an infinite family $\{\Ll_n=P^n_1\sqcup P^n_2\}_{n>0}$ of 2--component links of projective planes. Repeating the same procedure, we find a presentation
\begin{align*}\pi_1(S^4\setminus \Ll_n)&=\langle a,b\mid \overline{b}(\overline{a}\overline{b})^{n-1}a(ba)^n=\overline{a}(\overline{b}\overline{a})^{n-1}b(ab)^n=1\rangle\\&=\langle a,b\mid (\overline{a}\overline{b})^{n-1}a(ba)^n=a^2b^2=1\rangle\\&=\langle a,c\mid c=a\overline{c}\overline{a},c^n=\overline{a}c^n\overline{a}\rangle \hspace{.1in} (c=ab)\\&=\langle a,c\mid c=a\overline{c}\overline{a}, \overline{a}^2=c^{2n}\rangle\cong Q_{8n},\end{align*}
where $Q_{8n}$ is the generalized quaternion group of order $8n$.

The peripheral subgroup of $P^n_1$ inside $S^4\setminus\nu(\Ll_n)$ is generated by $a$ and
$$b(ab)^{n-1}=a^{-1}c^n,$$
so the peripheral subgroup of $P^n_1$ is generated by $a$ and $c^n$ and hence is isomorphic to $Q_8$ for all $n$. (Similarly, the peripheral subgroup of $P^n_2$ in $\pi_1(S^4\setminus \Ll_n)$ is isomorphic to $Q_8$.)

\begin{figure}{\centering

\labellist
  \pinlabel {\footnotesize $a$} at 10 28
   \pinlabel {\footnotesize $a$} at 46 28
    \pinlabel {\footnotesize $\overline{a}$} at 27 29.5
        \pinlabel {\footnotesize $\overline{a}$} at 63 29.5
          \pinlabel {\footnotesize $b$} at 100 30
   \pinlabel {\footnotesize $b$} at 136 30
    \pinlabel {\footnotesize $\overline{b}$} at 83 31
        \pinlabel {\footnotesize $\overline{b}$} at 119 31
        \pinlabel {\footnotesize $\overline{b}ab=\overline{a}$} at 315 130
         \pinlabel {\footnotesize $\overline{b}\overline{a}bab=\overline{b}$} at 353 152
          \pinlabel {\footnotesize $\overline{a}\overline{b}\overline{a}ba=a$} at 331 17
           \pinlabel {\footnotesize $\overline{a}\overline{b}a=b$} at 360 2
      \pinlabel \rotatebox{70}{$\textcolor{darkpurple}{b}$} at 354 105
\endlabellist
\includegraphics[width=5in]{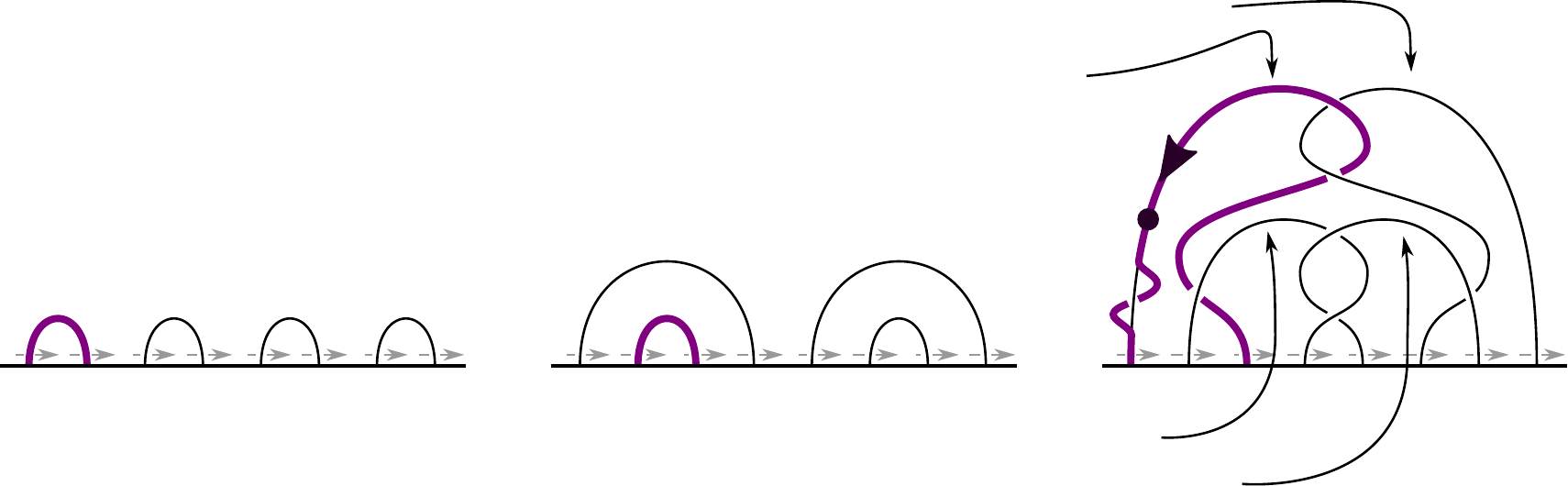}
\caption{A tri-plane diagram of $L=P_1\sqcup P_2$ as in Example \ref{ex:kinoshita}. We illustrate the process of obtaining a presentation of $\pi_1(S^4\setminus L)$. Near each bridge point, we draw arrows indicating an oriented meridian (which are labelled in the leftmost tangle diagram). In purple/bold, we indicate a generator of $\pi_1(P_1)$. Taking the basepoint to lie in $\partial(\nu(P_1))$, we follow the purple curve, passing under strands in the three twice via oriented meridians $a,b$ (in order). We push the curve off $P_1$, choosing framing so that the resulting curve does not link $P_1$. In this diagram, this yields the twist of the purple curve suggested in the rightmost piece of the tri-plane; this yields a curve representing $\overline{a}ab=b$ in $\pi_1(S^4\setminus L)$.}\label{fig:kinoshita}
}

\end{figure}
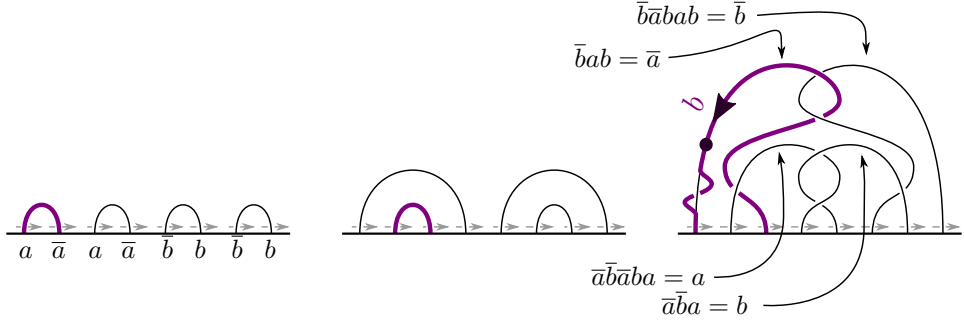

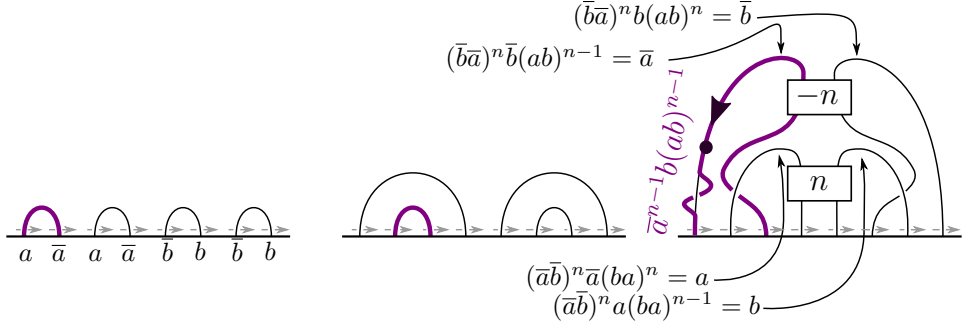
\begin{figure}{\centering

\labellist
  \pinlabel {\footnotesize $a$} at 10 28
   \pinlabel {\footnotesize $a$} at 46 28
    \pinlabel {\footnotesize $\overline{a}$} at 27 29.5
        \pinlabel {\footnotesize $\overline{a}$} at 63 29.5
          \pinlabel {\footnotesize $b$} at 100 30
   \pinlabel {\footnotesize $b$} at 136 30
    \pinlabel {\footnotesize $\overline{b}$} at 83 31
        \pinlabel {\footnotesize $\overline{b}$} at 119 31
          \pinlabel {$-n$} at 418 110
           \pinlabel {$n$} at 419 65
        \pinlabel {\footnotesize $(\overline{b}\overline{a})^{n}\overline{b}(ab)^{n-1}=\overline{a}$} at 280 130
         \pinlabel {\footnotesize $(\overline{b}\overline{a})^nb(ab)^n=\overline{b}$} at 338 152
          \pinlabel {\footnotesize $(\overline{a}\overline{b})^n\overline{a}(ba)^n=a$} at 315 17
           \pinlabel {\footnotesize $(\overline{a}\overline{b})^na(ba)^{n-1}=b$} at 335 2
      \pinlabel \rotatebox{80}{$\textcolor{darkpurple}{\overline{a}^{n-1}b(ab)^{n-1}}$} at 343 82
\endlabellist
\includegraphics[width=5in]{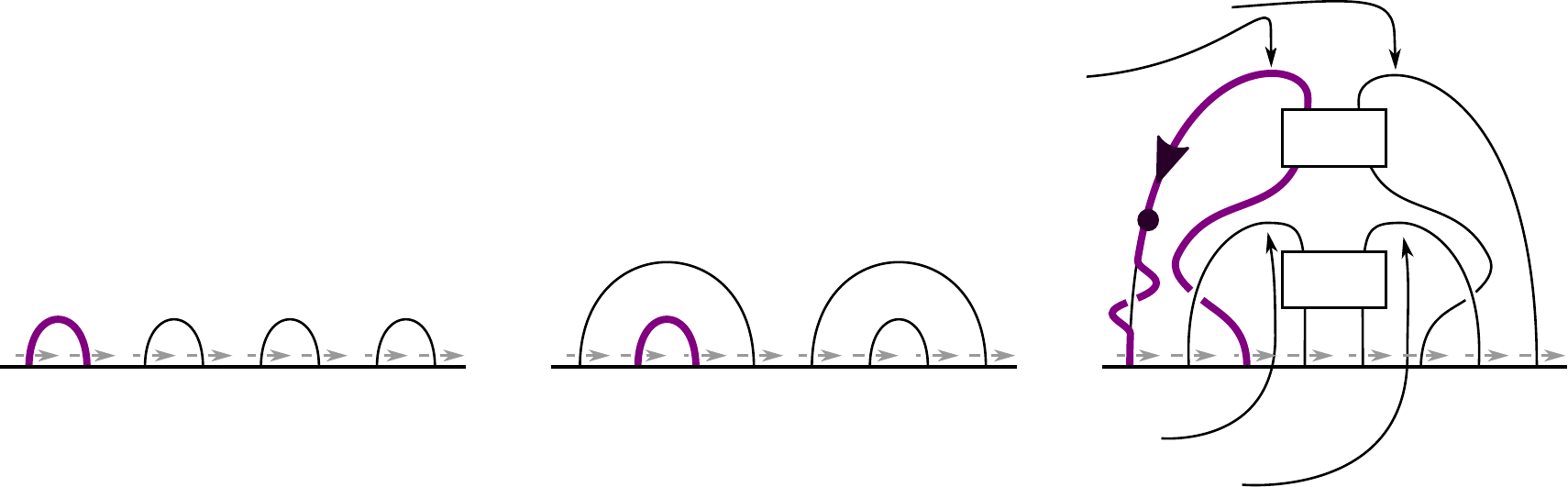}
\caption{A tri-plane diagram of $L_n=P^n_1\sqcup P^n_2$ as in Example \ref{ex:kinoshita}. We illustrate the process of obtaining a presentation of $\pi_1(S^4\setminus L_n)$. In purple/bold, we indicate a generator of $\pi_1(P^n_1)$. Near each bridge point, we draw arrows indicating an oriented meridian (which are labelled in the leftmost tangle diagram). Taking the basepoint to lie in $\partial(\nu(P^n_1))$, a parallel copy of this curve represents $\overline{a}^{n-1}b(ab)^{n-1}$.}\label{fig:kinoshita2}
}
\end{figure}

\end{example}

\subsection{Fox colorings and quandle colorings}

As in the classical Wirtinger algorithm, connected arcs of a diagram correspond to the same meridian of the knot group. Therefore, coloring the strands of a bridge trisection diagram with ``colors'' $0,1,\dots,p-1$ in such a way that at any crossing with overstrand $y$ and understrands $x$ and $z$ satisfies $c(x)+c(z)\cong 2c(y)\pmod p$, and so that the colors assigned to the points on the bridge sphere are the same in all three tangles encodes a Fox $p$-coloring of $K$. This has been observed in~\cite{Cahn-Kju}, and the connection with 3--fold covers is studied in~\cite{BCKM}.

The fundamental quandle $Q(S)$ of a knotted surface $\Ss$ in $S^4$ can be defined as the meridians of its knot group, under the new operation of conjugation. In other words, we define $x* y:=y^{-1}xy$. A presentation for the fundamental quandle is then obtained from the Wirtinger algorithm via this translation, and quandle colorings can be drawn diagrammatically on a tri-plane diagram as well. This has been studied in~\cite{Sato-Tan}, where the closely related kei colorings are used to give examples of knotted nonorientable surfaces with arbitrary bridge number.

\subsection{The Nielsen invariant of a bridge trisection}\label{subsec:Nielsen}

In this subsection we use Nielsen equivalence to distinguish certain ribbon bridge trisections of isotopic surfaces. Yasuda used Nielsen equivalence to distinguish ribbon presentations of the same 2--knot in \cite{yasuda}. Here we show that bridge trisecting those same ribbon presentations yields non-isotopic bridge trisections. Nielsen equivalence was also used by Islambouli to find inequivalent trisections of a closed 4-manifold of the same parameters \cite{gabe}.

Let $\mathcal{G}=(g_1,\dots,g_n)$ and $\mathcal{H}=(h_1,\dots,h_n)$ be two ordered lists of elements of a group $G$ such that each of the sets $\{g_1,\dots,g_n\}$ and $\{h_1,\dots,h_n\}$ generate $G$. If $\mathcal{H}$ can be obtained from $\mathcal{G}$ by a sequence of permutations, inverting elements, and replacing a generator $h_i$ with $h_i \cdot h_j$, $i\neq j$, then $\mathcal{G}$ and $\mathcal{H}$ are said to be {\it Nielsen equivalent}. Equivalently, if one thinks of $G$ and $H$ as constructed from $F_n$, the free group of rank $n$, as a quotient by normal subgroups $N_\mathcal{G}$ and $N_\mathcal{H}$, then $\mathcal{G}$ and $\mathcal{H}$ are Nielsen equivalent if and only if there is an automorphism $\phi$ of $F_n$ such that $\phi(G_i)=H_i$ for each $i$, where $G_i,H_i\in F_n$ such that $G_i \mod N_\mathcal{G}=g_i$ and $H_i \mod N_\mathcal{H}=h_i$.

Let $\TT$ be a bridge trisection and let $X_i=B^4\setminus \nu \mathcal{D}_i$ be the exterior of one of the trivial disk systems. Note that $B^4\setminus \nu \mathcal{D}_i\cong \natural^{c_i} S^1\times B^3$ is a 4-dimensional 1-handlebody. Choose any spine of $X_i$ and corresponding generators $(x_1,\dots,x_{c_i})$. The \textit{Nielsen class of $X_i$} is defined to be the Nielsen class of such a spine $(x_1,\dots,x_{c_i})$, denoted $\mathcal{N}(X_i)$. This is well-defined because any two spines are related by Nielsen transformations \cite{gabe}. Note that one can arrange that these generators $x_i$ are meridian elements for the trivial disk system, one for each component. Let $\phi_i:\pi_1(B^4\setminus \nu \mathcal{D}_i)\rightarrow \pi_1(S^4\setminus \nu \sS)$ be the (surjective) homomorphism induced by inclusion.

\begin{definition}
 Given a bridge trisection $\TT$ with disk system exteriors $X_i=B^4\setminus \nu \mathcal{D}_i$, let $\phi_i(\mathcal{N}(X_i))$ be the Nielsen class of  $\pi_1(S^4\setminus \nu S)$ induced by $\phi_i$. Then to the bridge trisection $\TT$ we associate the ordered triple of Nielsen classes $\mathcal{N}(\TT)=(\phi_1(\mathcal{N}(X_1)), \phi_2(\mathcal{N}(X_2)), \phi_3(\mathcal{N}(X_3)))$, which we call the {\bf Nielsen invariant of $\TT$}.
\end{definition}

To compute the Nielsen invariant of a bridge trisection $\TT$, first compute a presentation for $\pi_1(B^4\setminus \nu\DD_i)$. Then perform Reidemeister moves to obtain a crossingless unlink diagram, with generators expressed in terms of $\pi_1(B^4\setminus \nu\DD_i)$. Let $g_1,\dots,g_c$ denote one meridian for each component of this diagram. These are meridians to the minima of the disks, and hence form a spine of $X_i$. Then take $(g_1,\dots,g_c)$ as the Nielsen class of this disk system, and $\phi_i(\mathcal{N}(X_i)=(\phi_i(g_1)\dots,\phi_i(g_c))$.

\begin{proposition}
Let $\TT$ and $\TT'$ be bridge trisections. If $\TT$ is isotopic to $\TT'$, then $\mathcal{N}(\TT)=\mathcal{N}(\TT')$.
\end{proposition}

\begin{proof}
If $\TT$ is isotopic to $\TT'$, then there is an isotopy of $S^4$ taking each 4--ball-disk system $(B^4,\mathcal{D}_i)$ of $\TT$ to the corresponding pieces $(B^4,\mathcal{D}_i')$ of $\TT'$. Therefore, for each $i$, a spine of $X_i=B^4\setminus\nu\mathcal{D}_i$ is isotopic to a spine of $X_i=B^4\setminus\nu\mathcal{D}_i'$. As proven in \cite{gabe}, this implies the two spines are related by edge slides and orientation reversals, and hence their induced Nielsen classes are equivalent.
\end{proof}

A ribbon presentation $\mathfrak{R}$ induces a Wirtinger presentation for the knot group of the ribbon surface, with generating set a meridian for each component of the unlink $L=L_1\cup\cdots\cup L_n$ and one Wirtinger relation describing the linking of each tube with the unlink components. The induced Nielsen class $\mathcal{N}(\mathfrak{R})=(\mu_1,\dots,\mu_n)$ consists of these meridional generators.

\begin{proposition}
Let $\mathfrak{R}$ be a ribbon presentation of an orientable ribbon surface $\Ss$, with induced Nielsen class $\mathcal{N}(\mathfrak{R})$. Let $\TT$ be a bridge trisection of $\Ss$ induced by $\mathfrak{R}$. Then $\mathcal{N}(\TT)=(\mathcal{N}(\mathfrak{R}),\mathcal{N}(\mathfrak{R}),\mathcal{N}(\mathfrak{R}))$.
\end{proposition}

\begin{proof}
Let $\TT$ be a bridge trisection of $\Ss$ induced by $\mathfrak{R}$, via a virtual graph $G$ as in Section \ref{subsec:ribbon} (in particular, refer to Figure \ref{fig:ribbon_tri-plane}). Let $g_1,\dots,g_n$ denote meridians to the maxima of the unlink diagram $\DD_1\cup\overline{\DD}_{2}$, one for each vertical edge. Note that these form a spine for $X_1$, since we can isotope the diagram using the height function (pull the descending fingers back up to the top) to obtain a crossingless unlink diagram generated by $g_1\dots,g_n$: $\mathcal{N}(X_1)=(g_1,\dots,g_n)$. Similarly, for the unlink diagram $\DD_3\cup\overline{\DD}_{1}$, we take meridians $k_i$ to the minima, one for each vertical edge, and these form a spine by the same argument upside-down, thus $\mathcal{N}(X_3)=(k_1,\dots,k_n)$. Lastly, notice that the unlink diagram $\DD_2\cup\overline{\DD}_{3}$ is crossingless, and has one component for each of the vertical edges. Taking meridians $h_1,\dots,h_n$ to these components yields $\mathcal{N}(X_2)=(h_1,\dots,h_n)$.

The proof is complete once we recognize that $\phi_1(g_i)=\phi_2(h_i)=\phi_3(k_i)=\mu_i$, for then $\phi_i(\mathcal{N}(X_i))=\mathcal{N}(\mathfrak{R})$. This is the case because the vertical edges in the virtual graph correspond to the unlink components $L_i$, so the above-specified meridians are indeed meridians to the 2-spheres $L_i$. 
\end{proof}

\begin{corollary}
Let $\mathfrak{R}$ and $\mathfrak{R}'$ be two ribbon presentations of an orientable ribbon surface $\Ss$, with induced Nielsen classes $\mathcal{N}(\mathfrak{R})$ and $\mathcal{N}(\mathfrak{R}')$, and induced bridge trisections $\TT$ and $\TT'$. If $\TT$ is isotopic to $\TT'$, then $\mathcal{N}(\mathfrak{R})=\mathcal{N}(\mathfrak{R}')$.
\end{corollary}

\begin{remark} Recall the Schubert notation for a 2--bridge knot: let $\alpha,\beta$ be coprime integers with $\alpha>0$, $\beta$ odd, and $-\alpha<\beta<\alpha$. Schubert proved that the 2-bridge knot $S(\alpha,\beta)$ is equivalent to $S(\alpha^*,\beta^*)$ if and only if $\alpha=\alpha^*$ and $\beta = \beta^*$ or $\beta \beta^* \equiv 1 \mod 2\alpha$ \cite{schubert}. 
The Schubert notation indicates a particular bridge splitting of the knot $S(\alpha,\beta)$ with two minima. Taking meridians to the minima as generators, this induces a specific Nielsen class for the knot group $\pi_1(S^3\setminus S(\alpha,\beta))$.
F\"{u}ncke proved that if $\beta\beta^*\equiv 1 \mod 2\alpha$ and $\beta\neq\pm\beta^*$, then the induced Nielsen classes are inequivalent \cite{funcke}. Yasuda observed that spinning the knot $S(\alpha,\beta)$ by puncturing the knot at one of the maxima induces a ribbon presentation for $\text{Spin}(S(\alpha,\beta))$ with spheres corresponding to the minima and a tube corresponding to the remaining maximum \cite{yasuda}. Thus the Nielsen class induced by this ribbon presentation is the same as the one induced by the embedding of the original 2--bridge knot, yielding distinct ribbon presentations of the same spun 2-knots. The above corollary says that the bridge trisections induced by these ribbon presentations are also distinct. 
\end{remark}

\begin{corollary}
\label{cor:differentbridge}
There exist infinitely many ribbon 2--knots with pairs of bridge trisections $\TT$ and $\TT'$, both induced by ribbon presentations, which are non-isotopic as bridge trisections.
\end{corollary}

\begin{example}
As pointed out in \cite{yasuda}, $S(7,-3)$ and $S(7,-5)$ both present the knot $5_2$; thus the ribbon presentations induced by spinning these bridge splittings, as well as the induced bridge trisections are distinct. 

\end{example}

Stabilizing a surface by a trivial 1-handle stabilization does not change the group of its complement. If it is represented by a ribbon presentation, then it also does not change the induced Nielsen class. Thus by taking the connected sum of the above examples and any number of copies of the 3-bridge trisection of the unknotted torus, we obtain infinitely many pairs of orientable surface knots of any genus with inequivalent bridge trisections.

\begin{question}
If two ribbon presentations of a surface-knot are equivalent, must the bridge trisections induced by these ribbon presentations be isotopic?
\end{question}

\begin{question}
The three Nielsen classes induced by a ribbon bridge trisection are all equal. Does there exist a bridge trisection $\TT$ whose Nielsen invariant contains distinct Nielsen classes?
\end{question}

\bibliographystyle{amsalpha}
\bibliography{BT-facts}

\end{document}